\documentclass[a4paper,fleqn,12pt]{cas-sc}

\usepackage[numbers]{natbib}

\def\tsc#1{\csdef{#1}{\textsc{\lowercase{#1}}\xspace}}
\tsc{WGM}
\tsc{QE}

\usepackage{amsmath}
\usepackage{amsfonts}
\usepackage{amssymb}
\usepackage{amsthm}
\usepackage{enumerate}
\usepackage{mathtools}
\usepackage[ruled,vlined,boxed
]{algorithm2e}
\usepackage[dvipsnames]{xcolor}
\usepackage{graphicx}
\usepackage{tikz}
\usepackage{relsize}
\usepackage{stmaryrd}
\usepackage{bigints}
\usepackage[toc,page]{appendix}
\usepackage{graphicx}
\usepackage{float}
\usepackage{subcaption}

\usepackage{tikz}
\usepackage{pgfplots}

\usetikzlibrary{cd}

\newcommand{\N}{\mathbb{N}}

\newcommand{\Kt}{K \llbracket t \rrbracket}
\newcommand{\ktp}{k' \llbracket t \rrbracket}

\newcommand{\Ktzp}{K_0 \llbracket t \rrbracket [z]}

\newcommand{\Ktzzp}{K_0 \llbracket t \rrbracket \llbracket z \rrbracket}
\newcommand{\kt}{k \llbracket t \rrbracket}

\newcommand{\OKt}{\mathcal{O}_{K} \llbracket t  \rrbracket}
\newcommand{\OKtp}{\mathcal{O}_{K_0} \llbracket t  \rrbracket}

\newcommand{\OKtzp}{\mathcal{O}_{K_0} \llbracket t  \rrbracket [z]}

\newcommand{\OKtzzp}{\mathcal{O}_{K_0} \llbracket t \rrbracket \llbracket z \rrbracket}

\newcommand{\OK}{\mathcal{O}_{K}}

\newcommand{\OKp}{\mathcal{O}_{K_0}}
\newcommand{\Ktp}{K_0 \llbracket t \rrbracket}

\newcommand{\vertIII}[1]{{\left\vert\kern-0.25ex\left\vert\kern-0.25ex\left\vert #1
    \right\vert\kern-0.25ex\right\vert\kern-0.25ex\right\vert}}

\newcommand{\softO}{\tilde O}
\newcommand{\cA}{\text{\rm A}}
\newcommand{\cM}{\text{\rm M}}
\newcommand{\cMM}{\text{\rm MM}}
\newcommand{\cC}{\text{\rm C}}
\newcommand{\cD}{\text{\rm D}}

\begin{document}

\newtheorem{thm}{Theorem}
\newtheorem{lem}[thm]{Lemma}
\newtheorem{cor}[thm]{Corollary}
\newtheorem{propo}{Proposition}
\newtheorem{prop}[thm]{Proposition}
\newdefinition{remark}{Remark}
\newproof{pf}{Proof}
\newproof{pot}{Proof of Theorem \ref{thm}}
\newtheorem{thmx}{Theorem}
\newtheorem{defi}[thm]{Definition}
\newtheorem*{prf}{Proof}

\let\WriteBookmarks\relax
\def\floatpagepagefraction{1}
\def\textpagefraction{.001}

\shorttitle{Efficient computation of Cantor's division polynomials}    

\title [mode = title]{Efficient computation of Cantor's division polynomials of hyperelliptic curves over finite fields}

\author[]{Elie Eid}
                       

\affiliation[]{organization={IRMAR, Université de Rennes 1}, 
               country={France}}

\begin{abstract}[S U M M A R Y]
Let $p$ be an odd prime number. We propose an algorithm for computing rational representations of isogenies between Jacobians of hyperelliptic curves via $p$-adic differential equations with a sharp analysis of the loss of precision. Consequently, after having possibly lifted the problem in the $p$-adics, we derive fast algorithms for computing explicitly Cantor's division polynomials of hyperelliptic curves defined over finite fields. 
\end{abstract}



\maketitle



\section{Introduction}
\label{sec:introduction}

An important aspect in the study of principally polarized abelian varieties over finite fields is to design effective algorithms to calculate the number of points on these varieties. In 1985, Schoof proposed the first deterministic polynomial time algorithm for counting points on elliptic curves~\cite{Schoof85}. A few years later, improvements were made by computing kernels of isogenies, resulting in the Schoof-Elkies-Atkin algorihm which is sufficiently fast for practical purposes~\cite{Atkin,Morain95,Schoof95}. In 1990, Pila gave a generalization of the classical Schoof algorithm to abelian varieties and in particular Jacobians of curves over finite fields~\cite{Pila90}. His algorithm remains impractical in the general case but improvements were made for varieties of small dimension, typically for Jacobians of genus $2$ and $3$ curves~\cite{Gaudry00,Gaudry04,abelard}. 
When the inputs are Jacobians of hyperelliptic curves, isogenies (for curves of low genus) and Cantor's division polynomials (for curves of arbitrary genus) are important ingredients to these algorithms. For this reason, a keen interest has been raised to compute them efficiently~\cite{bomosasc08,lava16,couezo15,careidler20,Eid2021}. In this work, we tackle, in all generality, the problem of effective computation of isogenies between Jacobians of hyperelliptic curves to obtain fast algorithms that compute Cantor's division polynomials.


\subsection{Isogenies and ($p$-adic) differential equations}
\label{subsec:Isogeniesanddifferentialeq}

A separable isogeny between Jacobians of hyperelliptic curves of genus $g$
defined over a field $k$ is characterized by its so called rational
representation $($see Section~\ref{subsec:RationalRepresentation} for the
definition$)$; it is a compact writing of the isogeny
and can be expressed by $2g$ rational fractions defined over a finite
extension of $k$. These rational fractions are related. In fields of characteristic different from $2$, they can be determined by computing an
approximation of the solution $X(t) \in \ktp{}^{g}$, where $k'$ is a finite extension of $k$ of degree at most $O(g!)$, of a first order non-linear
system of differential equations of the form
\begin{equation}
\label{eq:nonlinearsystem}
H \left( X(t) \right) \cdot X'(t) = G(t)
\end{equation}
where $H \! : \ktp{}^{g} \rightarrow {M}_g\!\left( \ktp \right)$ is a well chosen map and $G(t) \! \in \!  \ktp{}^g$. This approach is a generalization of the elliptic curves case \cite{lava16} for which Equation~\eqref{eq:nonlinearsystem} is solved in dimension one. 

Equation~\eqref{eq:nonlinearsystem} was first introduced in \cite{couezo15}
for genus two curves defined over finite fields of odd characteristic and
solved in \cite{kieffer20} using a well-designed algorithm based on a Newton
iteration; this
allowed them to compute $X(t)$ modulo $t^{O({\ell})}$ in the case of an
$(\ell , \ell )$-isogeny for a cost of $\softO (\ell )$ operations in $k$ then recover the
rational fractions that defines the rational representation of the
isogeny. This approach does not work when the characteristic of $k$ is positive and small compared to $\ell$, in which case divisions by $p$ occur and an error
can be raised while doing the computations. We take on this issue similarly as
in the elliptic curve case \cite{lesi08, careidler20} by lifting the problem
to the $p$-adics. We will always suppose that the lifted Jacobians are also
Jacobians for some hyperelliptic curves. It is relevant to assume this, even
though it is not the generic case when $g$ is greater than $3$\footnote{Indeed, the dimension of the moduli scheme $\mathcal{M}_g$ is equal to $3g-3$, while the subspace of hyperelliptic curves in it has dimension $2g-1$.} \cite{oort1986},
since it allows us to compute efficiently the rational representation of the
multiplication by an integer in which case the lifting can be done arbitrarily. After this process, we need to analyze the loss of $p$-adic
precision in order to solve Equation~\eqref{eq:nonlinearsystem} without having a
numerical instability. We extend the result of \cite{lava16} to compute isogenies between Jacobians of hyperelliptic curves, by proving that
the number of lost digits when computing an approximation of the solution of Equation~\eqref{eq:nonlinearsystem} modulo $t^{O(g\ell)}$, stays within $O( \log _p( g \ell) )$ (see Sections~\ref{sec:GeneralCase} and \ref{sec:alternantdifferentialsystem}). 


\subsection{Computing with $p$-padic numbers}
\label{subsec:comoutationalmodel}
We introduce the computation model that we will use throughout this paper. Let $p$ be a prime number and $K$ a finite extension of the $p$-adic field $\mathbb{Q}_p$. We denote by $\upsilon _p$ the unique normalized extension to $K$ of the $p$-adic valuation. We denote by $\OK$ the ring of integers of $K$, $\pi \in \OK$ a fixed uniformizer of $K$ and $e$ the ramification index of the extension $K/\mathbb{Q}_p$. We naturally extend the valuation $\upsilon _ p$ to quotients of $\OK$, the resultant valuation is also denoted by $\upsilon _ p$.

From an algorithmic point of view, $p$-adic numbers behave like real numbers: they are defined as infinite sequences of digits that cannot be handled by computers. It is thus necessary to work with truncations. For this reason, several computational models were suggested to tackle these issues (see \cite{caruso17} for more details). In this paper, we use the \emph{fixed point arithmetic model} at precision $O(p^M)$, where $M \in \mathbb{N}^*$, to do computations in $K$. More precisely, an element in $K$ is represented by an interval of the form $a+ O(p^M)$ with $a \in \OK/\pi ^{eM}\OK$. We define basic arithmetic operations on intervals in an elementary way
\begin{align*}
\big( x + O(p^M) \big) \pm \big( y + O(p^M) \big) & = (x \pm y) + O(p^M)\,, \\
\big( x + O(p^M) \big) \times \big( y + O(p^M) \big) & = xy + O(p^M)\,.
\end{align*}
For divisions we make the following assumption: for $x, y \in \OK/\pi^{eM}\OK$, the division of $x + O(p^M)$ by $y + O(p^M)$ raises an error if $\upsilon _p(y) > \upsilon _p(x)$, returns $0 + O(p^M)$ if $x = 0$ in $\OK/\pi^{eM}\OK$ and returns any representative $z +O(p^M)$ with the property $x = yz$ in $\OK/\pi^{eM}\OK$ otherwise.

\subsubsection*{Matrix computation}
We extend the notion of intervals to the $K$-vector space $M_{n,m} (K)$: an element in
$M_{n,m}\! \left(K\right)$ of the form $A+ O(p^M)$ represents a matrix
$\left( a_{ij} + O(p^M) \right)_{ij}$ with
$A = (a_{ij}) \in M_{n,m}\! \left(\OK/\pi ^{eM}\OK \right)$. Operations in
$M_{n,m}\! \left(K \right)$ are defined from those in $K$:
\begin{align*}
\left( A + O(p^M) \right) \pm \left( B + O(p^M) \right) &= ( A \pm B  ) + O(p^M)\,,\\
\left( A + O(p^M) \right) \cdot \left( B + O(p^M) \right) &= ( A  \cdot  B  ) + O(p^M)\,.
\end{align*}
For inversions, we use standard Gaussian elimination.

\renewcommand*{\thepropo}{}
\begin{propo}\cite[Proposition~1.2.4 and Th\'eor\`eme~1.2.6]{vaccon15} 
\label{prop:Jordan}
Let $A \in \text{GL}_n(\mathcal{O}_K)$ with entries known up to precision $O(p^M)$. The Gauss-Jordan algorithm computes the inverse $A^{-1}$ of $A$ with entries known with the same precision as those of $A$ using $O(n^3)$ operations in $K$. 
\end{propo}


\subsection{Main result}
We are interested in designing fast algorithms that solve Equation~\eqref{eq:nonlinearsystem}. In a first step, we arrive at a generic algorithm for solving the differential system. Its complexity depends on the complexity of matrix multiplication and the composition $H(X(t))$. Let $\cMM (g, n)$ be the number of arithmetical operations required to compute the product
of two $g\times  g$ matrices containing polynomials of degree bounded by $n$. Our first theorem is the following.

\renewcommand*{\thethmx}{\Alph{thmx}}

\begin{thmx}[See Theorem~\ref{thm:mainthmbis} and Proposition~\ref{prop:complexity}]
\label{thm:main1}
Let $p$ be a prime number and $g \geq 1$ be an integer. Let $K$ be a finite extension of $\mathbb{Q}_p$ and $\OK$ be its ring of integers. There exists an algorithm that takes as input:
\begin{itemize}
\item two positive integers $n$ and $N$,
\item an analytic map $H \! : \OKt{}^{g} \rightarrow {M}_g\!\left( \OKt \right)$ of the form $H(y_1(t), \cdots , y_g(t)) = (f_{ij}(y_i(t)))_{ij}$ with $f_{ij} \in \OKt$ and $H(0) \in \text{GL}_g(\OK)$,
\item a vector $G(t) \in   \OKt{}^g$,
\end{itemize}
and, assuming that the unique solution of the differential equation
\begin{equation*}
H \left( X(t) \right) \cdot X'(t) = G(t)
\end{equation*}
is in $\left( t\OKt \right) ^g$, outputs an
approximation of this solution modulo $( p ^N , t^{n+1} )$ for a cost
\\ $O \left( \cMM(g,n) + \cC _H(n) \right)$ operations in $\OK$, where $\cC _H(n)$ denotes the algebraic complexity of an algorithm computing the composition $H(X(t)) \mod t^{n+1}$, at precision $O(p^M)$
with $M = \max ( N , 3) + \lfloor \log _ p (n) \rfloor$ if $p=2$,
$M=\max ( N , 2) + \lfloor \log _ p (n) \rfloor$ if $p=3$ and
$M= N + \lfloor \log _ p (n) \rfloor$ otherwise.
\end{thmx} 

It is important to know that one can do a bit better for $p=2$ and $3$ if we follow the same strategy as \cite{lava16}, in
this case $M$ is equal to $ \max ( N , 2) + \lfloor \log _ p (n) \rfloor$ if
$p=2$ and $ N+ \lfloor \log _ p (n) \rfloor$ otherwise. For the sake of simplicity, we will not prove this here. 

The field $K$ in the algorithm of Theorem~\ref{thm:main1} is generally given with a uniformizing parameter to avoid the calculation of its ring of integers. However, for the computation of isogenies, this ring is already known since the extension $K$ can always be chosen to be unramified.

If the integer $g$ is assumed to be small then the function $\cC _H(n)$ depends only on the number of arithmetical operations required to compute the composition of two power series modulo $t^{n+1}$. In the general case, Kedlaya and Umans bound~\cite{KU11} allows us to obtain $\cC _H(n) = \softO (n)$. In our context, the composition  $H(X(t)) \mod t^{n+1}$ is easy to compute since $H$ only includes univariate rational fractions of radicals of constant degree, therefore $\cC _H(n) = \cM(n)$, where $\cM (n)$ is the number of arithmetical operations required to compute the product of two polynomials of degrees bounded by $n$. Moreover, $ \cMM (g,n) = \cM(n)$, hence if $g$ is small, the algorithm in Theorem~\ref{thm:main1} requires at most $\softO ( n \cdot [ K : \mathbb Q _p])$ operations in $\mathbb Z_p$ to compute $X(t) \mod t^{n+1}$.

If $g$ is arbitrary then the algorithm of Theorem~\ref{thm:main1} requires at most $\softO(g^\omega  n \cdot [K : \mathbb Q_p])$ operations in $\mathbb Z_p$, where $\omega \in ]2,3]$ is a feasible exponent of matrix multiplication. This complexity can be reduced to $\softO(g n \cdot [K:\mathbb Q_p])$ operations in $\mathbb Z_p$ if $H$ is given by a structured matrix.\\
In the case of the computation of a rational representation of an isogeny $I$ over an unramified extension $K_0$ of $\mathbb Q_p$, the field $K$ can be chosen to be an extension of $K_0$ of degree at most $O(g)$ and $H$ is given by an alternant matrix. Therefore, the algorithm of Theorem~\ref{thm:main1} performs at most $\softO(g^2 n \cdot [K_0:\mathbb Q_p])$ operations in $\mathbb Z_p$ to compute an approximation of $X(t) \mod t^{n+1}$; but this is not optimal in $g$ since $I$ is defined over $K_0$. \\ In order to remedy this problem, we work directly on the first Mumford coordinate of a rational representation of $I$, \emph{i.e.} the degree $g$ monic polynomial whose roots are the components of the solution $X(t)$, which has the decisive advantage to be defined over the base field $K_0$. Consequently, we obtain a fast algorithm for computing a rational representation of $I$ in quasi-linear time. This is the main result of Section~\ref{sec:alternantdifferentialsystem}.

\begin{thmx}[See Theorem~\ref{thm:mainthmbis2} and Proposition~\ref{prop:complexityalternant}]
\label{thm:main2}
Let $p$ be an odd prime number. Let $K_0$ be an unramified extension of $\mathbb{Q}_p$ and $\OKp$ the ring of integers of $K_0$. There exists an algorithm that takes as input:
\begin{itemize}
\item three positive integers $g$, $n$ and $N$,
\item a monic polynomial $U_0 (z) = \prod \limits _{j=1}^g {(z - x_j^{(0)})} \in \OKp [z]$ such that $U_0(z) \mod p$ is separable,
\item a polynomial $V_0 \in \OKp [z]$ of degree $g-1$ such that $V(x_j^{(0)}) \mod {p} \neq 0$ for all $j \in \{ 1 , \ldots , g\} $,
\item a polynomial $f \in \OKtzp$ of degree $O(g)$ such that $U_0$ divides $f - V_0^2$ in $\OKtzp$,
\item a vector $G(t) \in   \OKtp{}^g$,
\end{itemize}
and, assuming that the unique solution $X(t)$ of the differential equation
\begin{equation*}
\left\{
\begin{array}{l}
H(X(t)) \cdot  X'(t) = G(t), \quad X(0) = (x_1^{(0)} , \cdots , x_g^{(0)}), \\
y_j(t)^2 = f(x_j(t)), \quad y_j(0) = V_0(x_j^{(0)}) \, \quad \text{for }j= 1, \ldots , g 
\end{array}
\right.
\end{equation*}
where $H(X(t))$ is the matrix defined by
\begin{equation*}
H(X(t)) = \left(\dfrac{x_j(t)^{i-1}}{y_j(t)} \right) _{1 \leq i,j \leq g },
\end{equation*} 
is in $\OKt ^g$, where $K$ denotes the splitting field of $U_0$, outputs a polynomial $U(t,z)= \prod \limits _{i=1}^g {(z - x_i(t))} \in \OKtp[z]$ such that $(x_1(t) , \ldots , x_g(t))$ is an approximation of this solution modulo $( p ^N , t^{n+1} )$ for a cost
 $\softO \left( n g \right)$ operations in $\OKp$, at precision $O(p ^{M+ \lfloor \log _p (2g-1) \rfloor})$, with $M= \max (N,2) + \lfloor \log _ p(n) \rfloor$ if $p=3$ and $M = N + \lfloor \log _ p(n) \rfloor$ otherwise.
\end{thmx}

\subsection{Computing Cantor's division polynomials}
Cantor's division polynomials are defined as being the numerators and denominators of the components of a rational representation of the multiplication by an integer endomorphism. They were first introduced for elliptic curves and later were described for hyperelliptic curves by Cantor~\cite{Cantor}. They are crucial in point counting algorithms
on elliptic and hyperelliptic curves. Classical algorithms for computing a rational representation of the multiplication endomorphism are usually based on Cantor's paper~\cite{Cantor} and Cantor's algorithm for adding points on Jacobians (see for example~\cite{abela18}). Although, they exhibit acceptable running time in practice, their theoretical complexity has not been well studied yet and experiments show that they
become much slower when the degree or the genus get higher. 

Using the algorithm of Theorem~\ref{thm:main2}, we derive a fast algorithm to compute Cantor's division polynomials over finite fields of odd characteristic. Our final result is then the following. 

\begin{thmx}[See Theorem~\ref{thm:multiplicationbyellmap}]
 Let $p$ an odd prime number and $g>1$ an integer. Let $\ell$ be an integer greater than $g$ and coprime to $p$. Let $C:y^2 = f(x)$ be a hyperelliptic curve of genus $g$ defined over a finite field $k$ of characteristic $p$. There exists an algorithm that computes Cantor $\ell$-division polynomials of $C$, performing at most $\softO(\ell^2  g^2 )$ operations in $k$.
 \end{thmx}



\section{Solving a system of $p$-adic differential equations: the general case}
\label{sec:GeneralCase}

In this section, we give a proof of Theorem~\ref{thm:main1} by solving the nonlinear system of differential equations \eqref{eq:nonlinearsystem} in an extension of $\mathbb{Q}_p$ for all prime numbers $p$. We use the computational model introduced in Section~\ref{subsec:comoutationalmodel} in our algorithm exposed in Section~\ref{subsec:algorithm} and the proof of its correctness is presented in Section~\ref{subsec:precisionanalysis}.\\
Throughout this section the letter $p$ refers to a fixed prime number and $K$ corresponds to a fixed finite extension of $\mathbb{Q}_p$. We denote by $\OK$ the ring of integers of $K$, $\pi \in \OK$ a fixed uniformizer and $e$ the ramification index of the extension $K/ \mathbb{Q}_p$.

\subsection{The algorithm}
\label{subsec:algorithm}
Let $g$ be a positive integer, $\Kt$ be the ring of formal series over $K$ in $t$. We denote by ${M}_g\! \left(k\right)$ the ring of square matrices of size $g$ over a field $k$. Let $\mathrm{f} = \big ( f_{ij} \big ) _{  i,j } \! \!  \in {M}_g \! \left ( \Kt \right)$ and $H_\mathrm{f}$ be the map defined by
\begin{equation*}
\begin{array}{rcl}
 \big ( t \Kt \big )^g  & \overset{H_\mathrm{f}} {\longrightarrow} & {M}_g \!  \left( \Kt \right)\, \\
\big ( x_1(t) , \ldots , x_g (t) \big ) & \longmapsto & \Big ( f_{ij} \big ( x_i(t)  \big ) \Big ) _{ij}\,.
\end{array}
\end{equation*}
Given $ \mathrm{f}  \in {M}_g \! \left( \Kt \right)$ and $ G = ( G_1, \ldots , G_g) \in  \Kt{} ^g$, we consider the following differential equation in $ X = (x_1,\ldots , x_g)$,
\begin{equation}
\label{eq:nonlinearequation}
H_\mathrm{f} \circ X \cdot X' = G.
\end{equation}
We will always look for solutions of \eqref{eq:nonlinearequation} in $\big (
t\Kt \big ) ^g$ in order to ensure that $H_\mathrm{f} \circ X$ is well
defined. We further assume that $H_\mathrm{f}(0)$ is invertible in ${M}_g \! \left( K \right)$.

The next proposition guarantees the existence and the uniqueness of a solution of the differential equation \eqref{eq:nonlinearequation}.

\begin{prop}
\label{prop:solutionexistence}
Assuming that $H_\mathrm{f}(0)$ is invertible in ${M}_g \!  \left( K \right)$, the system of differential equations \eqref{eq:nonlinearequation} admits a unique solution in $  \Kt{} ^g$.
\end{prop}

\begin{proof}
We are looking for a vector $X(t) = \sum  \limits _{n=1}^\infty {X_n t^n}$ that satisfies Equation~\eqref{eq:nonlinearequation}.
Since $X(0)= 0$ and $H_\mathrm{f} (0)$ is invertible in $ \Kt {} ^g$, then $H
_\mathrm{f} \big (X (t)\big )$ is invertible in $ M_{g} \!  \left( \Kt
\right)$. So Equation~\eqref{eq:nonlinearequation} can be written as
\begin{equation}
\label{eq:existanceproof}
 X '(t) = \big ( H _\mathrm{f}( X (t)) \big )^{-1} \cdot G(t).
\end{equation}
Equation~\eqref{eq:existanceproof} applied to $0$, gives the non-zero vector $X_1$. Taking the $n$-derivative of Equation~\eqref{eq:existanceproof} with respect to $t$ and applying the result to $0$, we observe that the coefficient $X_n$ only appears on the hand left side of the result, so each component of $X_n$ is a polynomial in the components of the $X_i$'s for $i<n$ with coefficients in $K$. Therefore, the coefficients $X_n$ exist and are all uniquely determined.
\end{proof}

We construct the solution of Equation~\eqref{eq:nonlinearequation} using a Newton scheme. We recall that for $Y = (y_1, \ldots ,y_g) \in \Kt {}^g$, the differential of $H_\mathrm{f}$ with respect to $Y$ is the function
\begin{equation}
\label{eq:chainformula}
\begin{array}{rcl}
dH_\mathrm{f} (Y) \, : \,  \Kt{} ^g & \longrightarrow & {M}_g \! \left( \Kt \right)\, \\
h & \longmapsto & dH_\mathrm{f}(Y)(h) = \left( { {f'_{ij}}} \left( y_i \right) \cdot  h_i \right)_{1\leq i,j \leq g}\,.
\end{array}
\end{equation}

We fix $m \in \N$ and we consider an approximation $X_m$ of $X$ modulo $t^m$. We want to find a vector $h \in \left( t^m \, \Kt \right)^g $, such that $X_m +h$ is a better approximation of $X$. We compute
\begin{equation*}
  H_\mathrm{f} \left( X_m + h \right)  =   H_\mathrm{f} \left( X_m \right)+ dH_\mathrm{f}(X_m) (h) \pmod {t^{2m}}\,.
\end{equation*}
Therefore we obtain the following relation
\begin{multline*}
 H_\mathrm{f} \left( X_m + h \right) \cdot \left( X_m + h \right) ' - G  = \\
H_\mathrm{f} \left( X_m \right) \cdot X_m' + H_\mathrm{f} \left( X_m \right) \cdot h' + dH_\mathrm{f}(X_m) (h) \cdot X_m' - G \pmod {t^{2m-1}}\,.
\end{multline*}
So we look for $h$ such that
\begin{equation}
\label{eq:newtoneq}
 H_\mathrm{f} \left( X_m  \right) \cdot h'  + dH_\mathrm{f}(X_m) (h) \cdot X_m'  =  -H_\mathrm{f} \left( X_m \right)\cdot X_m' +G \pmod {t^{2m-1}}\,.
\end{equation}
It is easy to see that the left hand side of Equation~\eqref{eq:newtoneq} is equal to $\left(H_\mathrm{f} \left( X_m \right) \cdot h \right)^{'}$, therefore integrating each component of Equation~\eqref{eq:newtoneq} and multiplying the result by $ \left(H_\mathrm{f} \left( X_m \right) \right)^{-1} $ gives the following expression for $h$
\begin{equation}
\label{eq:newtoniteration}
 h= \left(H_\mathrm{f} \left( X_m\right) \right) ^{-1} \,  \int {\left( G - H_\mathrm{f} \left( X_m \right)\cdot X_m' \right)\, dt} \pmod {t^{2m}},
\end{equation}
where $\int Y dt$, for $Y \in  \Kt {} ^g$, denotes the unique vector $I \in  \Kt {}^g$ such that $I' = Y$ and $I(0)=0$.\\
This formula defines a Newton operator for computing an approximation of the solution of Equation~\eqref{eq:nonlinearequation}. Reversing the above calculations leads to the following proposition.
\begin{prop}
\label{prop:newtoncase1}
We assume that $H_\mathrm{f}(0)$ is invertible in ${M}_g \! \left( K \right)$. Let $m\geq 0$ be an integer, $n=2m+1$ and $X_m \in  \Kt {}^g$ a solution of Equation~\eqref{eq:nonlinearequation} mod $t^{m+1}$. Then,
\begin{equation*}
X_n = X_m +  \left( H_\mathrm{f} \left( X_m \right) \right)^{-1} \int {\left(G - H_\mathrm{f} \left( X_m \right) \cdot X_m'\right)\, dt}
\end{equation*}
is a solution of Equation~\eqref{eq:nonlinearequation} mod $t^{n+1}$.
\end{prop}

It is straightforward to turn Proposition~\ref{prop:newtoncase1} into an algorithm that solves the non-linear system~\eqref{eq:nonlinearequation}. We make a small optimization by integrating the computation of $H_\mathrm{f}(X)^{-1}$ in the Newton scheme.

\begin{figure}[h!]
  \begin{center}
    \parbox{0.85\linewidth}{%
      \begin{footnotesize}\SetAlFnt{\small\sf}%
        \begin{algorithm}[H]%
          \caption{Differential Equation Solver} %
          \label{algo:DiffSolve}%
          \SetKwInOut{Input}{Input} %
          \SetKwInOut{Output}{Output} %
          \SetKwProg{DiffSolve}{\tt DiffSolve}{}{}%
          \DiffSolve{$(G,\mathrm{f},n)$}{
            \Input{$G, \mathrm{f} \mod {t^n}$ such that $H_\mathrm{f}(0)$ is invertible in $M_g \! \left( K \right)$. }
            \Output{The solution $X$ of Equation~\eqref{eq:nonlinearequation}$\mod {t^{n+1}}$, $H_\mathrm{f}\left(X \right) \mod t^{\lceil n/2 \rceil}$} \BlankLine %
            \If{$n =0$}
            {%
            \KwRet{$ 0 \mod t, \, H_\mathrm{f} (0)^{-1} \mod t$}
            }%
            {}
            $ m := \lceil \frac{n-1}{2} \rceil$\;%

            $X_m,\, H_m:= $ \texttt{DiffSolve}{$(G,\mathrm{f},m)$}\;%

            $H_n := 2H_m - H_m \cdot H_\mathrm{f}(X) \cdot H_m \mod t^{m+1}$

            \KwRet{$X_m +  H_n \mathlarger{\int} {\left(G - H_\mathrm{f} \left( X_m \right)\cdot X_m'\right)\, dt}\; \mod {t^{n+1}}$}
          }
        \end{algorithm}
      \end{footnotesize}
    }
  \end{center}
\end{figure}

According to Proposition~\ref{prop:newtoncase1},
Algorithm~\ref{algo:DiffSolve} runs correctly when its entries are given with an infinite $p$-adic precision; however it could stop working if we use the fixed point arithmetic model.  The next theorem
guarantees its correctness in this type of model.

\begin{thm}
\label{thm:mainthmbis}
Let $n,g \in \mathbb{N}$, $N\in \frac{1}{e}\mathbb{Z}^*, G \in \OKt ^g$ and
$\mathrm{f} \in M_g\! \left( \OKt\right)$. We assume that $H_\mathrm{f}(0)$ is
invertible in $M_g \! \left( \OK \right)$ and that the components of the
solution of Equation~\eqref{eq:nonlinearequation} have coefficients in $\OK$. Then, the procedure $\textrm{\tt DiffSolve}$ runs
with fixed point arithmetic at precision $O(p^M)$, with
$M = \max ( N , 3) + \lfloor \log _ p(n) \rfloor$ if $p=2$,
$M = \max ( N , 2) + \lfloor \log _ p(n) \rfloor$ if $p=3$ and
$M = N + \lfloor \log _ p(n) \rfloor$ otherwise, all the computations are done
in $\OK$ and the result is correct at precision $O(p^N)$.
\end{thm}

We give a proof of Theorem~\ref{thm:mainthmbis} at the end of Section~\ref{subsec:precisionanalysis}. Right now, we concentrate on the complexity of Algorithm~\ref{algo:DiffSolve}. Recall that $\cMM(g,n)$ is the number of arithmetical operations required to compute the product of two $g \times g$ matrices containing polynomials of degree $n$ and $\cM(n) := \cMM (1,n)$, therefore $\cM(n)$ is the number of arithmetical operations required to compute the product of two polynomials of degree $n$. According to \cite[Chapter~8]{Bostan17}, the two functions $\cM(.)$ and $\cMM(g,.)$ (in the worst case) are related by the following formula
\begin{equation}
\label{eq:relationcomplexity}
 \cMM(g , n) = O \left( g^{\omega}\cM(n) \right)
\end{equation}
where $\omega \in [2,3[$ is the exponent of matrix multiplication.
Furthermore, we recall that $\cC _{H} (n)$ denotes the algebraic complexity for
computing $H \circ X \mod t^n$ for an analytic map $H \! : \Kt{}^{g} \rightarrow {M}_g\!\left( \Kt \right)$ of the form $H = H_\mathrm{f}$ where $\mathrm{f} \in {M}_g\!\left( \Kt \right)$. We assume that $\cM (n)$ and $\cC _ H(n)$ satisfy the superadditivity hypothesis
\begin{equation}
\label{eq:superadditivity}
\begin{array}{cccc}
\cM (n_1 + n_2 ) & \geq &  \cM (n_1) + \cM (n_2),\\
 \cC_H (n_1 + n_2 ) & \geq &  \cC _H  (n_1) + \cC_H(n_2),
\end{array}
\end{equation}
 for all $n_1 , n_2 \in \mathbb{N}.$\\
Using Equation~\eqref{eq:relationcomplexity} we deduce the following relation
\begin{equation}
\label{eq:superadditivityMM}
\begin{array}{ccc}
O(\cMM (g,n_1 + n_2 ) )& \geq &  \cMM (g,n_1) + \cMM (g,n_2).
\end{array}
\end{equation}

\begin{prop}
\label{prop:complexity}
Algorithm~\ref{algo:DiffSolve} performs $O \left( \cMM( g,n ) + \cC _{H_\mathrm{f}}(n) \right)$ operations in $K$.
\end{prop}

\begin{proof}
Let $\cD$ denote the algebraic complexity of Algorithm~\ref{algo:DiffSolve}, then we have the following relation
\begin{equation*}
\cD (n)  \leq  \cD \left( \left\lceil \dfrac{n-1}{2} \right\rceil \right) + O \left( \cMM ( g , n ) + \cC _{H_\mathrm{f}}(n) \right).
\end{equation*}
Noticing that $g$ does not change at each iteration and using Equations~\eqref{eq:superadditivity} and \eqref{eq:superadditivityMM}, we find $ \cD (n) = O \left( \cMM (  g, n ) + \cC _{H_\mathrm{f}}(n) \right)$ and the result is proved.
\end{proof}

\begin{remark}
\label{remark:specialcase}
If the map $H_\mathrm{f}$ includes random univariate
rational fractions of radicals of constant degrees, the algebraic complexity $\cC_{H_\mathrm{f}}(n)$ is equal to $O\left( g^2 \cM(n) \right)$. Standard algorithms allow us to take $\cM(n) \in \softO (n)$. Therefore, Algorithm~\ref{algo:DiffSolve} outputs the solution of Equation~\eqref{eq:nonlinearequation} mod $t^{n+1}$ for a cost of $\softO (g^\omega n)$ operations in $\OK$.
\end{remark}

\begin{cor}
When performed with fixed point arithmetic at precision $O(p^M)$, the bit complexity of Algorithm~\ref{algo:DiffSolve} is $O \left(  \left(\cMM (g, n ) + \cC _{H_\mathrm{f}}(n) \right) \cdot \cA (K;M) \right)$ where $\cA (K;M)$ denotes an upper bound on the bit complexity of the arithmetic operations in $\OK / \pi ^{eM} \OK$.
\end{cor}

\subsection{Precision analysis}
\label{subsec:precisionanalysis}
The goal of this subsection is to prove Theorem~\ref{thm:mainthmbis}. The proof relies on the the theory of "differential precision" developed in \cite{carova14,carova15}.\\
We study the solution $X(t)$ of Equation~\eqref{eq:nonlinearequation} when $G(t)$ varies, with the assumption $H_\mathrm{f}(0)$ is invertible in ${M}_g \! \left( \OK \right)$. Proposition~\ref{prop:solutionexistence} showed that Equation~\eqref{eq:nonlinearequation} has a unique solution $X\! \left( G\right) \in  \Kt  ^g$. Moreover, if we examine the proof of Proposition~\ref{prop:solutionexistence}, we see that the $n+1$ first coefficients of the vector $X\! \left(G\right)$ depends only on the first $n$ coefficients of $G$. This gives a well-defined function
\begin{equation*}
\begin{array}{rcl}
X_n \, : \; \left( \Kt / \left( t^n \right) \right) ^g & \longrightarrow &  \left( t\Kt / \left( t^{n+1} \right) \right) ^g \\
G & \longmapsto & X\! \left(G \right)
\end{array}
\end{equation*}
for a given positive integer $n$. In addition, the proof of Proposition~\ref{prop:solutionexistence} states that for $G \in \left( \Kt / \left( t^n \right) \right) ^g$, $X_n \! \left(G \right)$ can be expressed as a polynomial in $ G(0) , G'(0), \ldots , G^{(n-1)}(0)$ with coefficients in $K$, therefore $X_n$ is locally analytic.
\begin{prop}
\label{prop:differentialX}
For $G \in  \left( \Kt / \left( t^n \right) \right) ^g$, the differential of $X_n$ with respect to $G$ is the following function
\begin{equation*}
\begin{array}{rcl}
dX_n \! \left(G \right) \, : \;  \left( \Kt / \left( t^n \right) \right) ^g & \longrightarrow &  \left( t\Kt / \left( t^{n+1} \right) \right) ^g \\
\delta G & \longmapsto & \left( H_\mathrm{f} \left( X_n \! \left(G\right) \right) \right) ^{-1} \cdot \mathlarger{\int} \delta G.
\end{array}
\end{equation*}
\end{prop}

\begin{proof}
We differentiate the equation $  H _ \mathrm{f} \! \left( X_n \!  \left( G \right) \right) \cdot X_n \! \left(G \right)' = G$ with respect to $G$. We obtain the following relation
\begin{equation}
\label{eq:proofdifferentialX}
H _ \mathrm{f} \! \left( X_n \! \left( G \right) \right) \cdot \big ( dX_n (G)(\delta G) \big )^{'} + d H_\mathrm{f} \! \left( X_n \! \left(G \right) \right)\!  \left(dX_n(G)(\delta G)\right) \cdot X_n(G)' = \delta G
\end{equation}
where $d H_\mathrm{f} \! \left( X_n(G) \right)$ is the differential of $H_\mathrm{f}$ at $X_n(G)$ defined in \eqref{eq:chainformula}. Making use of the relation
\begin{equation*}
\big ( \left( H _\mathrm{f} \! \left( X_n(G) \right) \right) \cdot dX_n(G)(\delta G) \big ) ^{'} = H _ \mathrm{f} \! \left( X_n \! \left( G \right) \right) \cdot \big ( dX_n (G)(\delta G) \big )^{'} + d H_\mathrm{f}\! \left( X_n(G) \right) \! \left(dX_n(G)(\delta G) \right) \cdot X_n(G)' ,
\end{equation*}
Equation~\eqref{eq:proofdifferentialX} becomes
\begin{equation*}
\big ( H_\mathrm{f}\!  \left( X_n(G)\right)  \cdot dX_n(G)(\delta G) \big ) ^{'} = \delta G.
\end{equation*}
Integrating the above relation and multiplying by $ \left( H_\mathrm{f} \! \left( X_n(G) \right) \right)^{-1}$ we get the result.
\end{proof}
We now introduce some norms on $\left( \Kt / \left( t^n \right) \right) ^g$ and $\left( t\Kt / \left( t^n \right) \right) ^g$. We set $ E_n = \left( \Kt / \left( t^n \right) \right) ^g$ and $F_n =\left( t\Kt / \left( t^{n+1} \right) \right) ^g$; for instance, $X_n$ is a function from $E_n$ to $F_n$.\\
First, we equip the vector space $K_n := \Kt / \left( t^n \right)$ with the usual Gauss norm
\begin{equation*}
 \Vert a_0 + a_1 t + \cdots + a_{n-1}t^{n-1} \Vert _{K_n} = \max \left( \left \vert a_0 \right \vert ,  \left \vert a_1 \right \vert, \ldots ,  \left \vert a_{n-1} \right \vert \right).
\end{equation*}
We endow $F_n$ with the norm obtained by the restriction of the induced norm $\Vert . \Vert$ on $F_n$: for every $X(t)=  \left( x_i(t) \right) _{i} \in F_n,$
\begin{equation*}
\left\Vert X(t) \right\Vert _ {F_n}  =  \max \limits _{i} \left\Vert x_i(t) \right\Vert _ {K_n} .
\end{equation*}
On the other hand, we endow $E_n$ with the following norm: for every $X(t) =  \left( x_i(t) \right) _i \in E_n,$
\begin{equation*}
\left\Vert X(t) \right \Vert _{E_n} =  \left\Vert \int X(t) \,  \right \Vert _{F_n} = \max \limits _{ i } \left \Vert  \int  x_i(t) \, \right \Vert _ {K_n}.
\end{equation*}
\begin{lem}
\label{lem:compatible}
Let $A \in {M}_g \!  \left( \OKt /(t^n) \right)$. If there exists a vector $x(t) \in  \left( \OKt / \left( t^n \right) \right) ^g$ such that $ \Vert A\,x \Vert _{F_n} < 1$ then $A$ is not invertible in ${M}_g \!  \left( \OKt /(t^n) \right)$.
\end{lem}

\begin{proof}
Write $A = (a_{ij} (t) )_{i,j}$ and $x(t) = (x_1(t) , \ldots , x_g(t))$. By definition, the norm $\Vert A \, x \Vert_{F_n}$ is equal to
\begin{equation*}
\Vert A \, x \Vert_{F_n} = \max \limits _i \left \Vert \sum \limits _j {a_{ij} x_j} \right \Vert_{K_n}.
\end{equation*}
Therefore, the condition $\Vert A \, x \Vert_{F_n} <1$ is equivalent to the following inequality
\begin{equation}
\label{eq:proofcompatible}
\left \Vert \sum \limits _j {a_{ij} x_j} \right \Vert_{K_n} <1
\end{equation}
for all $i = 1 , \ldots ,g$. Let $k$ be the residue field of $K$. Equation~\eqref{eq:proofcompatible} implies that $ \sum \limits _j {a_{ij} x_j} =0$ in $k$. Hence, the reduction of $A$ in ${M}_g \!  \left( \kt /(t^n) \right)$ is not invertible and $A$ is not invertible in  ${M}_g \!  \left( \OKt /(t^n) \right)$.
\end{proof}

\begin{lem}
\label{lem:isometry}
Let $G \in \left( \OKt / \left( t^n \right) \right) ^g$. We assume that $X_n \! \left(G \right) \in \left( t\OKt / \left( t^n \right) \right) ^g$, then $dX_n\! \left( G \right) \, : \; E_n \longrightarrow F_n$ is an isometry.
\end{lem}

\begin{proof}
The assumptions $X_n(G) \in \left( t\OKt / \left( t^n \right) \right) ^g$ and $H_\mathrm{f}(0) \in \text{GL}_g  \left (\OK \right )$ guarantee the invertibility of $H _\mathrm{f} \! \left( X_n(G) \right)$ in ${M}_g \!  \left( \OKt \right)$. Let $\delta G \in E_n$ such that $\Vert \delta G\Vert _{E_n} =1$. Using the fact that $H(X(G))^{-1} \int \delta G \in  \left( t\OKt / \left( t^n \right) \right) ^g$ and applying Lemma~\ref{lem:compatible}, we get
\begin{equation*} 
\Vert dX_n(G)(\delta G) \Vert = \Vert H(X(G))^{-1} \int \delta G  \Vert _{F_n} =1.
\end{equation*}
\end{proof}

We define the following function:
\begin{equation*}
\begin{array}{rcl}
\tau _ n \, : \, F_n \times E_n & \longrightarrow &  \text{Hom}(E_n, F_n)\\
( X \, , \, G) & \longmapsto & \left( \delta G \mapsto \left( H_\mathrm{f} \! \left( X \right) \right)^{-1} \cdot \mathlarger{\int} \delta G \right).
\end{array}
\end{equation*}

By Proposition~\ref{prop:differentialX}, the map $dX_n$ is equal to $ \tau _n \circ ( X_n , \text{id} )$, where id denotes the identity map on $E_n$.

\begin{lem}
\label{lem:Lambda}
Let $x \in \mathbb{R}$ such that $x < -2\dfrac{\log p}{p-1}$, then $\Lambda \! \left( X_n \right)_{\geq 2} \! (x)  < x$.
\end{lem}

\begin{proof}
One checks easily that $\Lambda (\text{id})(x) = x$ and, by Lemma~\ref{lem:Lambda}, $\Lambda (\tau _n )(x) \geq 0$ for all $x \in \mathbb{R}_+^*$. Applying \cite[Proposition~2.5]{carova15}, we get
\begin{equation*}
\Lambda \! \left( X_n \right)_{\geq 2} (x)  \leq  2 \! \left( x + \dfrac{\log p}{p-1}   \right)
\end{equation*}
 for all $x \leq - \dfrac{\log p}{p-1}  $. Therefore, $\Lambda  \! \left( X_n \right)_{\geq 2} \! (x)  < x$ if $ x <  -2\dfrac{\log p}{p-1}  $.
\end{proof}

\begin{prop}
\label{prop:precisionlemma}
Let $B_{E_n} \! \left( \delta \right)$ $($resp. $B_{F_n} \! \left( \delta \right) )$ be the closed ball in $E_n$ $($resp. in $F_n)$ of center $0$ and radius $\delta$. Under the assumption of Lemma~\ref{lem:isometry}, we have for all $ \delta <  p^{\frac{-2}{p-1}}$,
\begin{equation*}
X_n \! \left( G + B_{E_n} \! \left( \delta \right) \right) = X_n \! \left( G \right) + B_{F_n} \! \left( \delta \right).
\end{equation*}
\end{prop}

\begin{proof}
As a direct consequence of \cite[Proposition~3.12]{carova14} and Lemma~\ref{lem:Lambda}, we have the following formula
\begin{equation*}
 X_n \! \left( G + B_{E_n} \! \left( \delta \right) \right) = X_n \! \left( G \right) + dX_n \! \left(G\right) \! \left( B_{E_n} \! \left( \delta \right)\right),
 \end{equation*}
for all $\delta <  p^{\frac{-2}{p-1}}$.
The result follows from Lemma~\ref{lem:isometry}.
\end{proof}
We end this section by giving a proof of Theorem~\ref{thm:mainthmbis}.
\begin{proof}[Correctness proof of Theorem~\ref{thm:mainthmbis}]
Let $G, \mathrm{f}$ and $n$ be the input of Algorithm~\ref{algo:DiffSolve}. We first prove by induction on $n \geq 1$ the following equation
\begin{equation*}
H_\mathrm{f} \! \left( X_n \right) \cdot X_n^{'} = G \mod {(t^{n} , p^M)}.
\end{equation*}
 Let $m$ be a positive integer and $n = 2m+1$.
Let $e_ m = {G - H_\mathrm{f} \!  \left( X_m \right) \cdot  X'_m}$.
From the relation 
\begin{equation*}
 X_n = X_m + \left( H_\mathrm{f} \! \left( X_m \right) \right)^{-1}  \mathlarger{\int} {e_m\, dt}\; \mod {(t^{n+1} , p^M)}\,,
\end{equation*}
we derive the two formulas
\begin{equation}
\label{eq:proof1}
H_\mathrm{f} \! \left( X_m \right) \cdot X_n =  H_\mathrm{f} \! \left( X_m \right) \cdot X_m +  \mathlarger{\int} {e_m\, dt}\; \mod {(t^{n+1} , p^M)}
\end{equation}
and
\begin{align*}
H_\mathrm{f} \! \left( X_m \right) \cdot X'_n &=  H_\mathrm{f} \! \left( X_m \right) \cdot X'_m +  \left ( H_\mathrm{f} \! \left( X_m \right) \right)' \cdot \left( X_m - X_n\right) + e_m \mod {(t^{n} , p^M)}\,
\\ & = G + \left( H_\mathrm{f} \! \left( X_m \right)\right)' \cdot \left( X_m - X_n\right)\mod {(t^{n} , p^M)}\,
\\ & = G - \left( H_\mathrm{f} \! \left( X_m \right)\right) ' \cdot \left( H_\mathrm{f} \! \left( X_m \right) \right)^{-1}  \mathlarger{\int} {e_m\, dt \mod {(t^{n} , p^M)}}\,.
\end{align*}
Using the fact that the first $m$ coefficients of $e_m$ vanish, we get
\begin{align}
\label{eq:proof2}
H_\mathrm{f} \! \left( X_n \right)  \cdot X'_n  = H_\mathrm{f} \! \left( X_m \right) \cdot X'_n + dH _\mathrm{f} \! \left(X_m \right) \! \left( \left( H_\mathrm{f} \! \left( X_m \right) \right)^{-1}  \mathlarger{\int} {e_m\, dt} \right) \cdot X'_m \mod {(t^{n} , p^M)}\,.
\end{align}
In addition, one can easily verifies
\begin{equation*}
dH _\mathrm{f} \! \left(X_m \right) \! \left( \left( H_\mathrm{f} \! \left( X_m \right) \right)^{-1}  \mathlarger{\int} {e_m\, dt} \right)\cdot X'_m = \left( H_\mathrm{f} \! \left( X_m \right)\right)' \cdot \left( H_\mathrm{f} \! \left( X_m \right) \right)^{-1}  \mathlarger{\int} {e_m\, dt}
\end{equation*}
Hence, Equation~\eqref{eq:proof2} becomes
\begin{equation*}
H_\mathrm{f} \! \left( X_n \right)  \cdot X'_n = G \mod {(t^{n} , p^M)}.
\end{equation*}
Now, we define $ G_n = H_\mathrm{f} \! \left( X_n \right)  \cdot X'_n$ so that we have $X_n = X_n \! \left( G_n \right)$ and $\Vert G - G_n \Vert _{F_{n}} \leq p^{-M}$. Therefore, $ \Vert G - G_n \Vert _{E_{n}} \leq p^{-M+\lfloor \log _p(n) \rfloor}$. By Proposition~\ref{prop:precisionlemma}, we have that
\begin{equation*}
X_n \! \left( G_n \right) = X_n \! \left( G \right) \mod {(t^{n+1} , p^N)}.
\end{equation*}
Thus $ X_n = X_n \! \left( G \right) \mod {(t^{n+1} , p^N)}$.
\end{proof}



\section{Jacobians of curves and their isogenies}
\label{sec:Review}
Throughout this section, the letter $k$ refers to a fixed field of
characteristic different from two. Let $\bar{k}$ be a fixed
algebraic closure of
$k$. In Section~\ref{subsec:abelianvarieties}, we briefly recall some basic
elements about principally polarized abelian varieties and $(\ell , \ldots , \ell)$-isogenies between them; the notion of rational representation is discussed in Section~\ref{subsec:RationalRepresentation}. Finally, for a given rational representation, we construct a system of differential equations that we associate with it.


\subsection{$(\ell , \cdots , \ell)$-isogenies between abelian varieties}
\label{subsec:abelianvarieties}
Let $A$ be an abelian variety of dimension $g$ over $k$ and $A^{\vee }$ be its dual.
To a fixed line bundle $\mathcal{L}$ on $A$, we associate the morphism $\lambda _ \mathcal{L}$ defined as follows
\begin{equation*}
\begin{array}{rcl}
\lambda _\mathcal{L} \, : \:  A & \longrightarrow & A ^\vee  \\
x & \longmapsto & t_x ^* \mathcal{L} \otimes \mathcal{L}^{-1}
\end{array}
\end{equation*}
where $t_x$ denotes the translation by $x$ and $t_x^* \mathcal{L}$ is the pullback of $\mathcal{L}$ by $t_x$. \\
We recall from \cite{milne86} that an \textit{isogeny} between two abelian varieties is a surjective homomorphism of abelian varieties of finite kernel. The \textit{degree} of an isogeny is the number of preimages of a generic point in its codomain. \\
A \textit{polarization} $\lambda$ of $A$ is an isogeny $ \lambda : \, A \longrightarrow A^\vee$, such that over $\bar{k}$, $\lambda$ is of the form $\lambda _ \mathcal{L}$ for some ample line bundle $\mathcal{L}$ on $A_{\bar{k}}:= A \otimes \text{Spec}(\bar{k})$. 
When the degree of a polarization $\lambda$ of $A$ is equal to $1$, we say that $\lambda$ is a \textit{principal polarization} and the pair $(A, \lambda)$ is a \textit{principally polarized abelian variety}.
We assume in the rest of this subsection that we are given a principally polarized abelian variety $(A , \lambda)$. The \textit{Rosati involution} on the ring End$(A)$ of endomorphsims of $A$ corresponding to the polarization $\lambda$ is the map
\begin{equation*}
\begin{array}{rcl}
\text{ End}(A) & \longrightarrow & \text{End}(A) \\
\alpha & \longmapsto & \lambda ^{-1} \, \circ \alpha ^ \vee \circ \lambda.
\end{array}
\end{equation*}
The Rosati involution is crucial for the study of the division algebra End$(A) \otimes \mathbb{Q}$, but for our purpose, we only state the following result.

\begin{prop}\cite[Proposition~14.2]{milne86}
\label{prop:NS}
For every $\alpha \in$ End$\,\!(A)$ fixed by the Rosati involution, there exists, up to algebraic equivalence, a unique line bundle $ \mathcal{L} _ A  ^\alpha $ on $A$ such that $\lambda _{ \mathcal{L} _A ^\alpha } = \lambda \circ \alpha$. In particular, taking $\alpha$ to be the identity endomorphism denoted ``$1$'', there exists a
unique line bundle $\mathcal{L}_A ^1 $ such that $\lambda _{ \mathcal{L}_A ^1
} = \lambda $.
\end{prop}
The notion of \textit{algebraic equivalence} is defined as follows. We say that two line bundles $\mathcal{L}_1$ and $\mathcal{L}_2$ on $A$ are algebraically equivalent if they can be connected by a third line bundle, \textit{i.e.} if there exist a connected scheme $X$, two closed points $x_1,x_2 \in X$ and a line bundle $\mathcal{N}$ on $A \times X$, such that $ \mathcal{N}_{|{A \times\{x_1\}}} \simeq \mathcal{L}_1$ and $ \mathcal{N}_{|{A \times\{x_2\}}} \simeq \mathcal{L}_2$. We say that two divisors on $A$ are algebraically equivalent if their corresponding line bundles are. 

Using Proposition~\ref{prop:NS}, we give the definition of an $(\ell , \ldots , \ell)$-isogeny.
\begin{defi}
\label{def:Isogeny}
Let $(A_1 , \lambda _1)$ and $(A_2 , \lambda _2)$ be two principally polarized abelian varieties of dimension $g$ over $k$ and $\ell \in \mathbb{N}^*$. An $(\ell ,  \ldots , \ell )$-isogeny $I$ between $A_1$ and $A_2$ is an isogeny $I : \, A_1 \longrightarrow A_2$ such that
\begin{equation*}
 I ^* \mathcal{L}_{A_2} ^1 = \mathcal{L}_{A_1}^\ell,
 \end{equation*}
where $\mathcal{L}_{A_1}^\ell$ is the unique line bundle on $A_1$ associated with the multiplication by $\ell$ map.
\end{defi}

We now suppose that $A$ is the Jacobian of a  genus $g$ curve $C$ over $k$. We will always make the assumption that there is at least one $k$-rational point on $C$. Let $r$ be a positive integer and fix $P \in C$. We define $C^{(r)}$ to be the symmetric power of $C$ and $j_P^{(r)}$ to be the map

\begin{equation*}
\begin{array}{rcl}
 C^{(r)} & \longrightarrow & A \simeq J(C) \\
(P_1 , \ldots , P_r) & \longmapsto & [P_1 + \cdots P_r - r P].
\end{array}
\end{equation*}
If $r=1$ then the map $j_P^{(1)}$ is called the \textit{Jacobi map} with origin $P$.
We write $j_P$ for the map $j_P^{(1)}$.\\
The image of $j_P^{(r)}$ is a closed
subvariety of $A$ which can be also written as $r$ summands of
$j_P(C)$. Let $\Theta$ be the image of $j_P^{(g-1)}$, it is a divisor on $A$
and when $P$ is replaced by another point, $\Theta$ is replaced by a translate.  We call $\Theta$ the theta divisor associated with $A$.

\begin{remark}
\label{rem:LinebundleTheta}
If $A$ is the Jacobian of a curve $C$ and $\Theta$ its theta divisor, then
$\mathcal{L}_A^1 = \mathcal{L}( \Theta )$, where $\mathcal{L}( \Theta )$ is the sheaf associated to the divisor $\Theta$. 
\end{remark}

Using Remark~\ref{rem:LinebundleTheta}, Definition~\ref{def:Isogeny} for Jacobian varieties gives the following
\begin{prop}
\label{prop:Characterization}
Let $\ell \in \mathbb{N}^*$, $A_1$ and $A_2$ be the Jacobians of two 
algebraic curves over $k$ and $\Theta _1$ and $\Theta _2$ be the theta divisors associated to $A_1$ and $A_2$ respectively. If an isogeny $I \,: \, A_1 \longrightarrow A_2$ is an $(\ell , \ldots , 
\ell )$-isogeny then $I^* \Theta _2$ is algebraically equivalent to $ 
\ell \Theta _1$.
\end{prop}

\begin{proof}
For all $x \in A_1$, the theorem of squares \cite[Theorem~5.5]{milne86} gives the following relation
\begin{equation*}
 t_{\ell x} ^* \,  \mathcal{L}_{A_1}^1 \otimes \big ( \mathcal{L}_{A_1}^1 \big ) ^{-1} = \left ( t_{x} ^* \,  \mathcal{L}_{A_1}^1 \otimes \big ( \mathcal{L}_{A_1}^1 \big ) ^{-1} \right ) ^{\otimes \ell} = t_x ^* \big ( \mathcal{L} _ {A_1} ^1 \big ) ^{\otimes \ell} \otimes \big ( (\mathcal{L}_{A_1} ^1 )^{\otimes \ell \,  } \big ) ^{-1}.
\end{equation*}
Meaning that,
\begin{equation*}
\lambda _ {(\mathcal{L}_{A_1}^1)^{\otimes \ell } } =  \lambda _ {\mathcal{L}_{A_1}^\ell }.
\end{equation*}
From Proposition~\ref{prop:NS}, we deduce that the line bundle  $\mathcal{L}_{A_1} ^\ell$ is algebraically equivalent to $ \big ( \mathcal{L}_{A_1}^1 \big ) ^{\otimes \ell}$, therefore  $I^*\mathcal{L}_{A_2} ^1$ and $ \big ( \mathcal{L}_{A_1}^1 \big ) ^{\otimes \ell}$ are algebraically equivalent. By Remark~\ref{rem:LinebundleTheta},  $I^*\mathcal{L}_{A_2} ^1$ corresponds to $ I^* \Theta _2$ and $ \big ( \mathcal{L}_{A_1}^1 \big ) ^{\otimes \ell}$ corresponds to $ \ell \Theta _1$.
\end{proof}


\subsection{Rational representation of an isogeny between Jacobians of hyperelliptic curves}
\label{subsec:RationalRepresentation}
We focus on computing an isogeny between Jacobians of hyperelliptic curves.
Let $C_1$ $($resp. $C_2)$ be a genus $g$ hyperelliptic curve over $k$, $J_1$ $($resp. $J_2)$ be its associated Jacobian and $\Theta _1$ $($resp. $\Theta _2)$ be its theta divisor. We suppose that there exists a separable isogeny $I :  J_1 \longrightarrow J_2$. Let $P \in C_1$ be a Weierstrass point, let $j_P  : C_1 \longrightarrow J_1$ be the Jacobi map with origin $P$. Generalizing \cite[Proposition~4.1]{kieffer20} gives the following proposition

\begin{prop}
\label{prop:RationalRepresentation}
The morphism $I \circ j_P$ induces a unique morphism $I_P : \; C_1  \longrightarrow  C_2 ^{(g)} $ such that the following diagram commutes

\begin{center}
\begin{tikzcd}[column sep=large]
	& C_2 ^{(g)}  \\
  C_1 \arrow[ur, "I_P"] \arrow[dr, "I \circ j_P"'] &              \\
 & J_2 \arrow[uu ,"\simeq"']
        \end{tikzcd}
\end{center}

\end{prop}
$ $ \\
We assume that $C_1$ $($resp. $C_2)$ is given by the singular model
$v ^2 = f_1(u) \quad (\text{resp. } y^2 = f_2(x))$, where $f_1$ $($resp. $f_2)$ is a polynomial of degree $2g+1$ or $2g+2$. Set
$Q=(u,v) \in C_1$ and $I_P (Q) = \{(x_1, y_1) , \ldots,  (x_g, y_g)\}$. We use
the Mumford's coordinates to represent the
element $I_{P} (Q)$: it is given by a pair of polynomials $(U(X) , V(X))$ such that 
\begin{equation*}
U(X)= X^g + \mathbf{\sigma }_{1} X^{g-1}  + \cdots + \mathbf{\sigma}_g 
\end{equation*}
where
\begin{equation*}
\mathbf{\sigma} _i = (-1)^{i} \sum \limits _{1 \leq j_1 < j_2 < \cdots < j_i \leq g} { x_{j_1} x_{j_2} \cdots x_{j_i}}
\end{equation*}
and
\begin{equation*}
V(X) = \mathbf{\rho}_{1}X^{g-1} +  \cdots + \mathbf{\rho}_g = \sum \limits _{j = 0} ^{g-1} {y_j \left( \prod \limits _{i= 0 , i\neq j} ^{g-1}  \dfrac{X- x_i}{x_j - x_i} \right) }. 
\end{equation*}
The tuple $(\sigma _ 1, \cdots , \sigma _g , \rho _1 , \cdots , \rho _ {g})$ consists of rational functions (on $C_1$) in $u$ and $v$ and it is called a \textit{rational representation} of $I$. \\
We recall that the \textit{degree} of a rational function $f$ on a curve $C$, denoted by $\deg (f)$, is the number of its zeros (or poles). 

\begin{lem}
\label{lem:bornes}
Let $ \pi \, :  C_1 \rightarrow \mathbb{P}^1$ be a rational function on $C_1$.
\begin{enumerate}
\item If $\pi(u,v)$ is invariant under the hyperelliptic involution of $C_1$ then there exists a rational fraction $A$ in $u$ such that
\begin{equation*}
\pi (u,v) = A(u)
\end{equation*}
and $\deg (A) \leq \deg (\pi)/2$.
\item Otherwise, we can always find two rational fractions $B$ and $D$ in $u$ such that
\begin{equation*}
 \pi (u,v) = B(u) + v D(u)
\end{equation*}
and the degrees of $B$ and $D$ are bounded by $\deg (\pi) $ and $ \deg(\pi) + g +1$ respectively. Moreover, if $B(u)=0$ then $\deg(D) \leq   \deg(\pi)/2 + g +1$.
\end{enumerate}
\end{lem}

\begin{proof}
\begin{enumerate}
\item The inequality $\deg (A) \leq \deg (\pi)/2$ comes from the fact that the function $u$ has degree $2$.
\item The rational fractions $B(u)$ and $D(u)$ verify the following relations
\begin{equation*}
B(u) = \dfrac{\pi (u,v)+ \pi (u,-v)}{2}, \quad D(u) = \dfrac{\pi (u,v)- \pi (u,-v)}{2v}.
\end{equation*}
Since, $\pi (u,v)+ \pi (u,-v)$ and $\pi (u,v)- \pi (u,-v)$ are invariant under the hyperelliptic involution and have degrees bounded by $ 2\deg (\pi)$, then $B(u)$ is a rational fraction of degree bounded by $\deg (\pi)$ and $D(u)$  is a rational fraction of degree bounded by $\deg (\pi)+g+1$ (Note that $v$ is a rational fraction of degree bounded by $2g+2$).
\end{enumerate}
\end{proof}

\begin{prop}
\label{prop:degreerationalfractions}
The functions $\mathbf{\sigma}_1 , \ldots ,  \mathbf{\sigma}_{g}$ can be seen as rational fractions in $u$ and have the same degree bounded by $\deg (\sigma _ 1)/2$. Moreover, the rational functions $\mathbf{\rho}_1 /v  , \ldots , \mathbf{\rho}_{g} /v $ can also be expressed as rational fractions in $u$ of degrees bounded by $\deg ( \rho _1)/2 + g +1 , \ldots ,  \deg (\rho_{g})/2 +g+1$ respectively.
\end{prop}

\begin{proof}
It is a direct consequence of Lemma~\ref{lem:bornes} and using the fact that $I_{P} (u,-v) = -I_{P} (u,v)$.
\end{proof}

\begin{remark}
\label{remark:degreegeneralbound}
If $P$ is not a Weierstrass point, there exists rational fractions $A_i$,$B_i$,$D_i$ and $E_i$ in $u$ such that $\sigma _i(u,v) = A_i(u) + v B_i(u)$ and $\rho _i(u,v) = D_i(u) + v E_i(u)$ for all $i \in \{ 1 , \cdots , g\}.$ Let $\bar{P}$ the image of $P$ by the hyperelliptic involution. The morphism $I_{\bar{P}}$ gives a rational representation $(\overline{\sigma _ {1}}, \cdots , \overline{\sigma _g} , \overline{\rho _1} , \cdots , \overline{\rho _ {g}})$ of $I$. From the relation $I_P(u, -v) = - I_{\bar{P}} (u,v)$, we deduce
$\overline{\sigma _i}(u,v) = A_i(u) - v B_i(u)$ and $\overline{\rho _i}(u,v) = -D_i(u) + v E_i(u)$ for all $i \in \{ 1 , \cdots , g\}.$ This gives the following formulas
\begin{equation*}
A_i(u) =({\sigma_i(u,v) + \overline{\sigma _i}(u , v)})/{2}, \quad  B_i(u) =( {\sigma_i(u,v) - \overline{\sigma _i}(u , v)})/{2v},
\end{equation*}
\begin{equation*}
D_i(u) = ({\rho_i(u,v) - \overline{\rho _i}(u , v)})/{2}, \quad  E_i(u) = ({\rho_i(u,v) + \overline{\rho _i}(u , v)})/{2v}.
\end{equation*}
The degrees of $A_i$ and $D_i$ (resp. $B_i$ and $E_i$) are bounded by $\deg (\sigma_i)$ (resp. $\deg ( \rho _i) +g+1$).
\end{remark}

In order to determine the isogeny $I$, it suffices to compute its rational representation (because $I$ is a group homomorphism), so we need to have some bounds on the degrees of the rational functions $\mathbf{\sigma}_1 , \ldots ,  \mathbf{\sigma}_{g}, \mathbf{\rho}_1  , \ldots , \mathbf{\rho}_{g}  $. In the case of an $(\ell , \ldots , \ell )$-isogeny, we adapt the proof of \cite[\S~6.1]{couezo15} in order to obtain bounds in terms of $\ell$ and $g$.

\begin{lem}
\label{lem:polardivisor}
Let $i \in \{ 1 , \ldots , g\}$. The pole divisor of  $\sigma _i$ seen as
function on $J_2$, is algebraically equivalent to $2 \Theta _ 2$. The pole divisor of $\rho _i$ seen as function
on $J_2$ is algebraically equivalent to $3 \Theta _ 2$ if $\deg
(f_2) = 2g+1$, and $4 \Theta _2$ otherwise.
\end{lem}

\begin{proof}
  This is a generalization of \cite[Lemma~4.25]{kieffer20}. Note that if
  $\deg (f_2) = 2g+1$, then $\sigma _ i$ has a pole of order two along the
  divisor $\{ ( R_1 , \ldots , R_{g-1} , \infty ) \, ; R_i \in C_2 \}$ which is algebraically equivalent to $ \Theta _2$.
\end{proof}

\begin{lem}[{\cite[Appendix]{matsusaka59}}]
\label{lem:Theta}
The divisor $j_P(C_1)$ of $J_1$ is algebraically equivalent to
$\dfrac{\Theta _ 1^{g-1}}{(g-1)!}$ where $\Theta _1 ^{g-1}$ denotes the
$g-1$ times self intersection of the divisor $\Theta _1$.
\end{lem}

\begin{prop}
\label{prop:degreegl}
Let $\ell$ be a non-zero positive integer and $i \in \{ 1 , \ldots , g\}$. If
$I$ is an $(\ell , \ldots , \ell)$-isogeny, then the degree of $\sigma _i$
seen as a function on $C_1$ is bounded by $2g \ell$. The degree of $\rho _i$ seen as a function on $C_1$ is bounded by $3g \ell$ if $\deg (f_2) = 2g+1$, and $4g \ell$ otherwise.
\end{prop}

\begin{proof}
The degrees of $\sigma _ 1 , \ldots , \sigma _ g , \rho _1 , \ldots , \rho
_ g$ are obtained by computing the intersection of $I_P(C)$ with their pole
divisors. By Lemma~\ref{lem:polardivisor}, it suffices to show that
\begin{equation*} 
I_P(C) \cdot \Theta _ 2 = \ell g.
\end{equation*}
Since $I$ is an $(\ell , \ldots , \ell)$-isogeny, Proposition~\ref{prop:Characterization} gives that $I^* \Theta _ 2$ is algebraically equivalent to $\ell \Theta _ 1$. Moreover, up to algebraic equivalence,
\begin{equation*}
I^* \big ( I_P ( C) \big ) =  \big ( \vert \ker ( I) \vert \big )  \,    j_P(C) = \ell ^g j_P(C).
\end{equation*}
Using Lemma~\ref{lem:Theta}, we obtain
\begin{equation*}
I^* \big ( I_P(C) \big ) \cdot I^* \Theta _2  = g\ell^{g+1}.
\end{equation*}
As 
\begin{equation*}
I^* \big ( I_P(C) \big ) \cdot I^* \Theta _2 =  \deg (I)  \, \big (  I_P(C) \cdot \Theta _2 \big ) = \ell^g (  I_P(C) \cdot \Theta _2 \big ),
\end{equation*}
 the result follows.
\end{proof}

\subsection{Associated differential equation}
\label{subsec:diffeq}
We assume that char$(k) \neq 2$. We generalize \cite[\S~6.2]{couezo15} by constructing a differential system modeling the map $F_P= I \circ j_P $ of Proposition~\ref{prop:RationalRepresentation}.
The map $F_P$ is a morphism of varieties, it acts naturally on the spaces of holomorphic differentials $H ^0 ( J_2 ,  \Omega^1 _{J_2} )$ and $ H ^0 ( C_1 ,  \Omega^1 _{C_1} )$ associated to $J_2$ and $C_1$ respectively, this action gives a map 
\begin{equation*}
F_P ^*  \, : \, H ^0 ( J_2 ,  \Omega^1 _{J_2} ) \longrightarrow H ^0 ( C_1 ,  \Omega^1 _{C_1} ).
\end{equation*}
A basis of $H ^0 ( C_1 ,  \Omega^1 _{C_1} )$ is given by
\begin{equation*}
B_1 = \left\{ u ^i \dfrac{du}{v} \, ; i \in \{ 0,\ldots , g-1\} \right\} .
\end{equation*}
The Jacobi map of $C_2$ induces an isomorphism between
the spaces of holomorphic differentials associated to $C_2$ and $J_2$, so $H
^0 ( J_2 ,  \Omega^1 _{J_2} )$ is of dimension $g$, it can be identified with
the space $ H ^0 ( C_2^g ,  \Omega^1 _{C_2^g} ) ^ {S_n}$ (here the symmetric group $S_n$ acts naturally on the space $ H ^0 ( C_2^g ,  \Omega^1 _{C_2^g} )$). With this identification, a basis of $ H ^0 ( J_2 ,  \Omega^1 _{J_2} )$ is chosen to be equal to
\begin{equation*}
B_2 = \left\{  \sum \limits _{j=1}^ g { x_j ^i \dfrac{dx_j}{y_j}} \, ; \, i \in \{ 0, \ldots , g-1\} \right\}.
\end{equation*}
Let $(m_{ij})_{0 \leq i,j \leq g} \in  \text{ GL}_g(\bar{k})$ be the matrix of $F_P^*$ with respect of these two bases, we call it the \textit{normalization matrix}.

\begin{remark}
Let $P_1$ and $P_2$ be two points on $C_1$. The two morphisms $I_{P_1}$ and $I_{P_2}$ satisfy the following relation
\begin{equation*} 
I_{P_1} = I_{P_2} + I([P_2 - P_1]).
\end{equation*}
Therefore, the linear maps $I_{P_1}^*$ and $I_{P_2}^*$ are equal. 
\end{remark}

Let $Q=(u_Q,v_Q) \in C_1$ be a non-Weierstrass point different from $P$ and
$I_{P} (Q) = \{R_1 ,\ldots , R_g\}$ such that $I_P(Q)$ contains $g$ distinct points and does not contain neither a point at infinity nor a Weierstrass point.
The points $R_i$ may be defined over an extension $k'$ of $k$ of degree equal to $O(g!)$. Let $t$ be a formal parameter of $C_1$ at $Q$, then we have the following diagram

\begin{center}
\begin{tikzcd}[column sep=large]
	\text{Spec} \big ( k' \llbracket t \rrbracket \big ) \arrow[rr, "t \mapsto ( R_i(t))_i"] \arrow[dd] & & C_2 ^g  \arrow[dd]  \\
	 & &\\
  C_1 \arrow[rr, "I_P"] & & C_2 ^{(g)}

\end{tikzcd}
\end{center}
For all $i= 1 , \ldots ,g$, the pull back of $\sum \limits _{j=1}^{i-1} {x^{i-1}dx_j/y_j}$ along the bottom horizontal arrow, then along the left vertical arrow, gives 
\begin{equation*}
\frac{du}{v}\sum \limits _{j=1}^g{m_{ij}u^{j-1}}.
\end{equation*}
And the pull back of $\sum \limits _{j=1}^{i-1} {x^{i-1}dx_j/y_j}$ along the right vertical arrow, then along the top horizontal arrow gives 
\begin{equation*}
 \sum \limits _{j=1}^g {x_j^{i-1}dx_j}.
\end{equation*}
This gives the differential system
\begin{equation}
\label{eq:diffsys}
\left \{
\begin{array}{ccccccc}
\dfrac{dx_1}{y_1}& + & \cdots & + & \dfrac{dx_g}{y_g} & = & \big ( {m_{11} + m_{12} \cdot u + ... + m_{1g}\cdot u^{g-1}} \big ) \dfrac{du}{v}\,, \\
& &  & & &  &  \\
\dfrac{x_1 \cdot dx_1}{y_1}& + & \cdots & + & \dfrac{x_g \cdot dx_g}{y_g} & = & \big ( {m_{21} + m_{22} \cdot u +  ... + m_{2g}\cdot u^{g-1}} \big ) \dfrac{du}{v}\,,  \\
& & \vdots & & &  & \vdots  \\
\dfrac{x_1 ^{g-1} \cdot dx_1}{y_1}& + & \cdots & + & \dfrac{x_g^{g-1} \cdot dx_g}{y_g} & = & \big ( {m_{g1} + m_{g2} \cdot u +  ... + m_{gg}\cdot u^{g-1}} \big ) \dfrac{du}{v}\,, \bigskip \\
y_1^2 = f_2(x_1), & & \cdots & &, \, y_g^2 = f_2(x_g)\,. &
\end{array}
\right.
\end{equation}

Equation~\eqref{eq:diffsys} has been initially constructed and solved in \cite{couezo15} for $g = 2$. In this case, the normalization matrix and the initial condition $(x_1(0) , x_2(0))$ are computed using algebraic theta functions.
In a more practical way, we refer to \cite{kieffer20} for an easy computation of the initial condition $(x_1(0) , x_2(0))$ of Equation~\eqref{eq:diffsys} and for solving the differential system using a Newton iteration. However, in this case, the normalization matrix is determined by differentiating modular equations.
There is a slight difference in Equation~\eqref{eq:diffsys} between the two
cases, especially $x_1(0)$ and $x_2(0)$ are different in the first, and equal in the second. Let $H$ be the $g$-squared matrix defined by 
\begin{equation*}
H(x_1,\ldots x_g) = \left( x_j^{i-1}\dfrac{1}{y_j} \right) _{1 \leq i,j \leq g }.
\end{equation*}
We suppose that $g=2$. If the initial condition $(x_1(0), x_2(0))$ of Equation~\eqref{eq:diffsys} satisfies $x_1(0) \neq x_2(0)$, then the matrix $H ( x_1(0), x_2(0)  )$ is invertible in ${M}_2 \! \left( k' \right)$. Otherwise, its determinant is equal to zero. \\
More generally, we prove that with the assumptions that we made on $Q,R_1,R_2, \ldots R_{g-1}$ and $R_g$, the matrix $H ( x_1(0), \ldots ,  x_g(0)  )$ is invertible in ${M}_g \! \left( k' \right)$. Let $t$ be a formal parameter, $Q(t)=(u(t), v(t))$ the formal point on $C_1 \left( k\llbracket t \rrbracket \right)$ that corresponds to $t = u - u_Q$ and $\{R_1(t) , \ldots , R_g(t)\}$ the image of $Q(t)$ by $I_{P}$, then Equation~\eqref{eq:diffsys} becomes
\begin{equation}
\label{eq:diffsys1}
H \! \left( X(t) \right) \cdot  X'(t) = G(t)
\end{equation}
where $X(t) = (x_1 (t) , \ldots  , x_g(t))$ and $G(t) = v ^{-1}\left( \sum \limits _{i= 1}^g {m_{ij} {u^{i-1}}} \right) _ {1 \leq j \leq g}$. Thus we have the following proposition.
\begin{prop}
\label{prop:Determinant}
The matrix $H \! \left( X(t) \right)$ is invertible in ${M}_g \! \left( k' \llbracket t \rrbracket \right)$.
\end{prop}

\begin{proof}
The matrix $H \! \left( X(t) \right)$ is an alternant matrix, its determinant is given by
\begin{equation*}
\det \left( H \! \left( X(t) \right) \right) = \dfrac{\prod \limits _{ 1 \leq i < j \leq g} {\left( x_j(t) - x_i(t) \right)} }{\prod \limits _{i=1} ^ g {y_i (t) } }
\end{equation*}
which is invertible in ${M}_g \! \left( k' \llbracket t \rrbracket \right)$ because $x_i(0) \neq x_j(0)$  for all $i,j \in \{ 1 , \ldots , g\}$ such that $i\neq j$.
\end{proof}

\begin{cor}
\label{cor:isogenyassumptions}
Let $p$ be a prime number. We assume that $k$ is an extension of $\mathbb{Q}_p$. Up to a change of variables, Equation~\eqref{eq:diffsys1} fulfills all the assumptions of Equation~\eqref{eq:nonlinearequation}, in particular it admits a unique solution in $k' \llbracket t \rrbracket$.
\end{cor}
By Corollary~\ref{cor:isogenyassumptions}, it is straightforward to make use of Algorithm~\ref{algo:DiffSolve} to solve Equation~\eqref{eq:diffsys1} when $k$ is an unramified extension of the field of $p$-adic numbers. This gives rise to an algorithm that computes a rational representation of a given $(\ell , \cdots , \ell)$-isogeny between Jacobians of hyprelliptic curves of genus $g$, whose complexity is quasi-optimal with respect to $\ell$ but not in $g$. Thus, Algorithm~\ref{algo:DiffSolve} can only be used efficiently to compute isogenies of Jacobians of hyperelliptic curves of small genus.



\section{Solving alternant systems of differential equations}
\label{sec:alternantdifferentialsystem}
Let $p$ be an odd prime number. In this section, we aim for effective resolution of Equation~\eqref{eq:diffsys1} when it is defined over an unramified extension of $\mathbb{Q}_p$. We re-examine the Newton scheme of Algorithm~\ref{algo:DiffSolve} to make it quasi-linear in the dimension of the solution $X(t)$. This will give a quasi-optimal algorithm to compute rational representations of isogenies between Jacobians of hyperelliptic curves over finite fields, after having possibly lifted the problem in the $p$-adics. 

The three next subsections are concerned with preliminary material: we introduce the differential system that we want to solve, then we recall some computational results that will eventually be used in our main algorithm exposed in Section~\ref{subsec:Alternantdifferentialsystem}.


\subsection{The setup}
\label{subsec:setup}
We keep the same notation as Section~\ref{subsec:diffeq} and we assume that $p \neq 2$ and $k$ is a finite field of characteristic $p$. For $i \in \{ 1 , \ldots , g\}$, write $R_i = (x_i^{(0)}, y_i^{(0)})$. We recall that the computation of the rational representation associated with $I_P$ reduces to the problem of computing an approximation of the following differential system whose unknown is $X(t) = (x_1(t) , \ldots , x_g(t)) \in k' $.
\begin{equation}
\label{eq:mainsystem}
\left\{
\begin{array}{l}
H(X(t)) \cdot  X'(t) = G(t), \quad X(0) = (x_1^{(0)} , \cdots , x_g^{(0)}), \\
y_j(t)^2 = f_2(x_j(t)), \quad y_j(0) = y_j^{(0)} \, \quad \text{for }j= 1, \ldots , g 
\end{array}
\right.
\end{equation}
where $G(t) \in \kt ^g $ and $H(X(t))$ are the matrices defined by
\begin{equation*}
\begin{array}{ccc}
G(t) = \dfrac{1}{v(t)}\left( \sum \limits _{i= 1}^g {m_{ij} {u(t)^{i-1}}} \right) _ {1 \leq j \leq g}
&
\text{and}
&
H(X(t)) = \left(\dfrac{x_j(t)^{i-1}}{y_j(t)} \right) _{1 \leq i,j \leq g }.
\end{array}
\end{equation*}
Based on the discussion in Section~\ref{subsec:Isogeniesanddifferentialeq}, we are sometimes obliged to lift Equation~\eqref{eq:mainsystem} to the $p$-adics. Therefore, we will replace $k$ by an unramified extension $K_0$ of $\mathbb{Q}_p$ and $k'$ by an unramified extension $K$ of $K_0$ of degree at most $O(g)$. Consequently, $f_1$, $f_2$ and the components of $G(t)$ have coefficients in $\OKp$; moreover, $X(t) \in \OKt ^g$. 

By Corollary~\ref{cor:isogenyassumptions}, Equation~\eqref{eq:mainsystem} can be solved using the following Newton iteration
\begin{equation*}
X_{2m+1}(t) = X_m(t) + H(X_m(t))^{-1} \int (G - H(X_m(t))\cdot X'_m(t)) \, dt.
\end{equation*}
Or, equivalently, 
\begin{equation}
\label{eq:Newtoniteration}
H(X_m(t)) \cdot ( X_{2m+1}(t) -  X_m(t))  = \int (G - H(X_m(t))\cdot X'_m(t)) \, dt.
\end{equation}

A call from Algorithm~\ref{algo:DiffSolve} gives the desired result, but this is not optimal in $g$. As explained in Section~\ref{subsec:Isogeniesanddifferentialeq}, this lack of efficiency is due to the fact that the components of the solution $X(t)$ of Equation~\eqref{eq:mainsystem} have coefficients defined over the field $K$, whose degree over $K_0$ depends on $g$. For this reason, we work directly on the first Mumford polynomial
\begin{equation*}
U(t,z) = \prod \limits _{j=1}^g {(z - x_j(t))}
\end{equation*}
whose coefficients are defined over the ring $\Ktp$: we rewrite the Newton scheme~\eqref{eq:Newtoniteration} accordingly and design fast algorithms for iterating it in quasi-linear time.


\subsection{Computing Newton sums}
We recall an efficient algorithm for the computation of the Newton sums of a polynomial. Let $K_0$ be an unramified extension of $\mathbb{Q}_p$. Let $P(t,z)$ be a monic polynomial of degree $d$ with coefficients in $\Ktp$ such that $P(0,z)$ is separable over $K_0$ and $x_1(t), x_2(t) , \cdots , x_d(t)$ its roots in $\Kt$, where $K$ denotes the splitting field of $P(0,z)$. We define the \emph{$i$-th Newton sum} $s_i(t)$ of $P$ by 
\begin{equation*}
s_i(t) = \sum \limits _{j=1}^d  {{x_j(t)}^i}  \in \Ktp, 
\end{equation*}
and we are interested in designing an efficient algorithm to compute it, only from the coefficients of $P$.\\
 Let $P^*$ be the reciprocal polynomial of $P$, \textit{i.e.} $P^*(t,z)= \prod \limits _{j=1}^d {(1-x_j(t)z)}$. It is well known that the $i$-th coefficient of the power series expansion of $ (P^*)'/P^*$ in $\Ktzzp$ is equal to $- s_{i+1}(t) $ (see for instance~\cite[Lemma~2]{NewtonSums}). This gives Algorithm~\ref{algo:NewtonSums} to compute the first $g$ Newton sums of the polynomial $P(t,z)$ modulo $t^{n+1}$.

\begin{figure}[h!]
  \begin{center}
    \parbox{0.8\linewidth}{%
      \begin{footnotesize}\SetAlFnt{\small\sf}%
        \begin{algorithm}[H]%
          \caption{Newton Sums} %
          \label{algo:NewtonSums}%
          \SetKwInOut{Input}{Input} %
          \SetKwInOut{Output}{Output} %
          \SetKwProg{NewtonSums}{\tt NewtonSums}{}{}%
          \NewtonSums{$(P,g,n)$}{
            \Input{$P \in \Ktp[z]$, $g \in \mathbb{N}^*$, $n \in \mathbb{N}^*$.}
            \Output{The sequence $s_1(t) \mod t^{n+1}, \cdots , s_g(t) \mod t^{n+1}$.} \BlankLine %
            {}
            $f := -(P^*)' / P^* \mod (t^{n+1},z^{g})$  \tcp*{$f = \sum \limits _{i=0}^{g-1} f_i(t) z^i$}
            \KwRet{$f_0(t),\cdots , f_{g-1}(t)$}
          }
        \end{algorithm}
      \end{footnotesize}
    }
  \end{center}
\end{figure}

\begin{prop}
\label{prop:complexitynewtonsums}
Let $P \in \OKtp[z]$ be a monic polynomial and $g,n,N \in \mathbb{N}^*$. When the procedure {\tt NewtonSums} runs with fixed point arithmetic at precision $O(p^N)$, all the computations are done in $\OKp$ and the result is correct at precision $O(p^N)$. Moreover, Algorithm~\ref{algo:NewtonSums} performs at most $\softO(ng)$ operations in $K_0$.
\end{prop}

\begin{proof}
The fact that all the computations stay within $\OKp$ is a direct consequence of the assumption $P \in \OKtp[z]$ and the fact that $P^*$ is invertible in $\OKtzzp$. In addition, it is easy to see that the output of {\tt NewtonSums}$(P,g,n)$ is correct at precision $O(p^N)$.\\ The inverse power series of $P^*$ modulo $(t^{n+1},z^{g})$ in $\Ktzzp$ is computed by the Newton iteration $Q \mapsto Q(2- Q P^*)$. Therefore, the complexity of the computation of $f$ only depends on the complexity of multiplying two bivariate polynomials of total degree $n+g$. This can be done using at most $\softO(gn)$ operations in $K_0$.
\end{proof}


\subsection{Hankel matrix-vector product}
\label{subsec:Hankel}
Let $g \in \mathbb{N}^*$. We recall that a $ g \times g$ Hankel matrix $A$ is a $g \times g$ matrix of the form
\begin{equation*}
A = \begin{pmatrix}
a_0 & a_1 & a_2 & \cdots & \cdots & a_{g-1} \\
a_1 & a_2 & & \cdots & \cdots & a_{g}\\
a_2 & & & & & \vdots \\
\vdots & & & & & \vdots\\
a_{g-1} &\cdots  & \cdots & a_{2g-4} & a_{2g-3} & a_{2g-2}
\end{pmatrix}.
\end{equation*}
Matrix-vector multiplication for this type of matrix can be computed in $O(M(g))$ arithmetic operations instead of $O(g^\omega)$, where $\omega \in ]2,3]$ is a feasible exponent of matrix multiplication.
\begin{prop} 
\label{prop:Hankelproduct} 
Let $n \in \mathbb{N}$. Let $K_0$ be an unramified extension of $\mathbb{Q}_p$. Let $A=(a_{i+j-2}(t))_{i,j} \in M_g(\Ktp)$ be a Hankel matrix and $v = (v_1(t), \cdots , v_g(t)) \in \Ktp^g$. Let $f$ and $h$ be the two polynomials
\begin{equation*}
f(t,z) = a_0(t) + a_1(t) z + a_2(t) z^2 + \cdots + a_{2g-3}(t)z^{2g-3}+ a_{2g-2} (t)z^{2g-2}
\end{equation*}
 and
\begin{equation*}
h(t,z) = v_1(t) z^{g-1} + v_2(t) z^{g-2} + \cdots + v_g(t).
\end{equation*}
Write $f(t,z) \cdot  h(t,z) \mod (t^{n+1}, z^{2g-1}) = \sum \limits _{i=0}^{2g-2} {w_i(t) z^i} $ then 
\begin{equation*}
A \cdot  v \! \mod t^{n+1} = (w_{g-1}, \cdots, w_{2g-2}).
\end{equation*}
\end{prop}
\begin{proof}
Let $i \in \{ 1 , \ldots , g\}$. The $(g-2+i)$-th coefficient of the product $R = f(t,z) \cdot  h(t,z)$ is equal to 
\begin{equation*}
R_{g-2+i} (t)= \sum \limits _{j=i-1}^{g-2+i} {a_{j}(t)v_{j+2-i}(t)} = \sum \limits _{j=1}^{g} {a_{i+j-2}(t)v_{j}(t)}.
\end{equation*}
Therefore, $R_{g-2+i} (t)$ is the $i$-th component of the product $A \cdot v$. 
\end{proof}

Proposition~\ref{prop:Hankelproduct} gives a quasi-linear algorithm to compute the matrix-vector product for Hankel matrices.

\begin{figure}[h!]
  \begin{center}
    \parbox{0.8\linewidth}{%
      \begin{footnotesize}\SetAlFnt{\small\sf}%
        \begin{algorithm}[H]%
          \caption{Hankel matrix-vector product} %
          \label{algo:HankelProduct}%
          \SetKwInOut{Input}{Input} %
          \SetKwInOut{Output}{Output} %
          \SetKwProg{HankelProd}{\tt HankelProd}{}{}%
          \HankelProd{$(A,v,n)$}{
            \Input{$A =(a_{i+j-2}(t))_{i,j}  \mod t^{n+1}$, $v = (v_1(t) , \cdots, v_g(t)) \mod t^{n+1}$, $n \in \mathbb{N}$.}
            \Output{The product $A \cdot v \mod t^{n+1}$.} \BlankLine %
            {}
            $ f := a_0 + a_1 z + a_2 z^2 + \cdots + a_{2g-3}z^{2g-3}+ a_{2g-2} z^{2g-2}$\;%
            $h:=v_1 z^{g-1} + v_2 z^{g-2} + \cdots + v_g$\;%
            $w := f h \mod (t^{n+1}, z^{2g-1})$  \tcp*{$w= \sum \limits _{i=0}^{2g-2} w_iz^i$}

            \KwRet{$w_{g-1}, \cdots, w_{2g-2}$}
          }
        \end{algorithm}
      \end{footnotesize}
    }
  \end{center}
\end{figure}

\begin{prop}
\label{prop:Hakelcomplexity}
Let $n \in \mathbb{N}$. Let $K_0$ be an unramified extension of $\mathbb{Q}_p$. Let $A=(a_{i+j-2}(t))_{i,j} \in M_g(\Ktp)$ be a Hankel matrix and $v = (v_1(t), \cdots , v_g(t)) \in \Ktp^g$. When it is called on the input $(A,v,n)$, the algorithm {\tt HankelProd} performs at most $\softO(ng)$ operations in $K_0$. 
\end{prop}

\begin{proof}
This is a direct consequence of Proposition~\ref{prop:Hankelproduct}. 
\end{proof}


\subsection{The alternant differential system}
\label{subsec:Alternantdifferentialsystem}
We go back to the system of differential equations~\eqref{eq:mainsystem}. We will make use of Equation~\eqref{eq:Newtoniteration} to construct its solution by successive approximations. 

Let $m \geq 0$ be an integer and $n = 2m+1$. We suppose that we are given an approximation $U_m(t,z)= \prod \limits _{j=1} ^g {(z -x_j^{(m)}(t))}\in \Ktzp$ of the polynomial $U(t,z)$ modulo $t^{m+1}$, such that the vector $X_m=(x_1^{(m)}(t), \ldots x_g^{(m)}(t))$ satisfies Equation~\eqref{eq:mainsystem} modulo $t^m$. In order to compute an approximation $U_n(t,z)$ of $U(t,z)$ modulo $t^{n+1}$ from $U_m(t,z)$, we use Equation~\eqref{eq:Newtoniteration} and perform the three following steps:

\begin{enumerate}
\item Compute $H(X_m(t))\cdot X'_m(t)$ modulo $t^{n+1}$.
\item Compute $F_m (t) =  \int (G - H(X_m(t))\cdot X'_m(t)) \, dt. \mod t^{n+1}$
\item Compute $U_n$ by solving linear system $H(X_m(t)) \cdot (X_n(t) - X_m(t))= F_m(t)$ modulo $t^{n+1}$.
\end{enumerate}

Once step~1 is carried out, step~2 can be executed using at most $\softO(ng)$ operations in $K_0$, because the components of the vector $H(X_m(t))\cdot X'_m(t)$ are defined over $K_0$. 
The following construction will show that the vector $X_m$ will not be of any use to perform steps~1 and 3 but only the polynomial $U_m(t,z)$.

Write $f_2(z) = \sum \limits _{j=1}^{g-1} {f_j z^j}$. Let $s_i^{(m)}(t) \in \Ktp$ be the $i$-th Newton sum of $U_m(t,z)$ and 
\begin{equation}
\label{eq:ri}
r_i^{(m)} (t) = \dfrac{1}{i} \dfrac{ds_i^{(m)}(t)}{dt} = \sum \limits _{j=1} ^g \dfrac{dx_j^{(m)}(t)}{dt}\cdot {x_j^{(m)}(t)}^{i-1} \in \Ktp .
\end{equation}
Let $W_m(t,z) = \sum \limits _{i=0}^{g-1} w_i^{(m)}(t) \,z^i \in \Ktzp$ be the degree $g-1$ polynomial such that
\begin{equation*}
W_m(t,z)^2  =   \dfrac{1}{f_2(z)} \mod (t^{n+1}, U_m(t,z)) \\
\end{equation*} 
with the initial condition $W_m(0,x_j ^{(0)})  = y_j^{(0)},$ for all $j \in \{ 1 , \ldots , g\}.$\\
Set $V_m (t,z) = f_2(z) \cdot  W_m (t,z) \mod (t^{n+1}, U_m(t,z))$ and, for $ j \in \{1, \ldots , g\}$, let $y_j^{(m)} (t)$ be the power series $V_m ( t , x_j^{(m)}(t))$. By construction, we have 
\begin{equation*}
W_m ( t , x_j^{(m)} (t) ) \cdot y_j^{(m)}(t) \equiv 1 \mod t^{n+1}
\end{equation*} 
and 
\begin{equation*}
y_j^{(m)} (t) \equiv y_j(t) \mod t^{m+1}, 
\end{equation*} 
for all $ j \in \{1, \ldots , g\}$.
\begin{prop}
\label{prop:HankelproductHX}
The product $H(X_m(t))\cdot X'_m(t)$ satisfy the following relation:
\begin{equation*}
H(X_m(t))\cdot X'_m(t) \mod {t^{n+1}} =
\begin{pmatrix}
r_1^{(m)} & r_2^{(m)} & \cdots & r_g^{(m)}\\
r_2^{(m)} & r_3^{(m)} &  & r_{g+1}^{(m)} \\
\vdots \\
r_{g}^{(m)} & r_{g+1}^{(m)} & \cdots & r_{2g-1}^{(m)}
\end{pmatrix}
\begin{pmatrix}
w_0(t) \\
w_1(t) \\
\vdots \\
w_{g-1}(t)
\end{pmatrix}
\end{equation*}
\end{prop}

\begin{proof}

This is a direct consequence of the fact that  
$ {1}/{y_j^{(m)}(t)} \mod t^{n+1} = \sum \limits _{i=0} ^{g-1}{w_i^{(m)} (t) x_j^{(m)}(t) ^{i} }$, for all $j \in \{ 1 , \ldots , g\}$.
\end{proof}

\begin{cor}
Step~1 can be carried out using at most $\softO(ng)$ operations in $K_0$.
\end{cor}

\begin{proof}
Approximations of the Newton sums $s_i^{(m)}(t)$ modulo $t^{n+1}$ can be computed from the polynomial $U_m(t,z)$ and using Algorithm~\ref{algo:NewtonSums}. By Proposition~\ref{prop:complexitynewtonsums}, this can be done using at most $\softO(ng)$ operations in $K_0$. 
The power series $r_i^{(m)} (t)$ can be computed using Equation~\eqref{eq:ri} and the polynomial $W_m(t,z)$ is constructed from the classical Newton scheme for extracting square roots.\\
By Proposition~\ref{prop:HankelproductHX}, the product $H(X_m(t))\cdot X'_m(t) \mod t^{n+1}$ is a Hankel matrix-vector product that can be computed using Algorithm~\ref{algo:HankelProduct}. By Proposition~\ref{prop:Hakelcomplexity}, this can be done using at most $\softO(ng)$ operations in $K_0$. 
\end{proof}

We now explain how to compute $U_n$ from the linear system $H(X_m(t)) \cdot (X_n(t)- X_m(t)) = F_m(t)$ modulo $t^{n+1}$. The following proposition and its proof give a construction over $\Ktp$ of an approximation of the interpolating polynomial of the data $\{ (x_1^{(m)}(t), x_1^{(n)}(t)- x_1^{(m)}(t)), \ldots ,(x_g^{(m)}(t), x_g^{(n)}(t)- x_g^{(m)}(t) )\}$ modulo $t^{n+1}$.

\begin{prop}
\label{prop:interpolatingpolynomial}
Let $m \geq 0$ be an integer and $n= 2m+1$. There exists a polynomial $h_m(t,z) \in \OKt[z]$ of degree $g-1$ such that $x_j^{(n)}(t) = h_m(t, x_j^{(m)}(t))+ x_j^{(m)}(t) \mod t^{n+1}$ for all $j \in \{ 1 , \ldots , g\}$.
\end{prop}

\begin{proof}
We give a construction of $h_m$ using the approach in~\cite[Section~5]{Kaltofen}. Let $D_m(t,z)$ be the polynomial defined by 
\begin{equation*}
D_m(t,z) = x_1^{(m)}(t)z^g + x_2^{(m)}(t) z^{g-1} + \cdots + x_{g-1}^{(m)}(t) z^2 + x_g^{(m)}(t) z.
\end{equation*}
Write
\begin{equation*}
U_m(t,z)  \cdot D_m(t,z) = q_{2g}^{(m)}(t) z^{2g} + q_{2g-1}^{(m)}(t)z^{2g-1} + \cdots + q_1^{(m)}(t) z+ q_0^{(m)}(t) 
\end{equation*}
and 
\begin{equation*}
 Q_m(t,z)= q_{2g}^{(m)}(t) z^{g-1} + q_{2g-1 }^{(m)} (t)z^{g-2} + \cdots + q_{g+2}^{(m)}(t) z + q_{g+1}^{(m)}(t) .
\end{equation*}
For all $j = 1, \cdots , g$ we have the following relation
\begin{equation*}
x_j^{(n)}(t)- x_j^{(m)}(t)= \dfrac{Q_m(t,x_j^{(m)}(t))}{\partial _z U_m (t,x_j^{(m)}(t)) }\cdot y_j^{(m)}(t),
\end{equation*}
where $\partial _z U_m$ denotes the partial derivative of $U_m$ with respect to variable $z$. 
Hence, we take $h_m$ to be equal to  
\begin{equation*}
h_m (t,z) = \dfrac{Q_m(t,z)}{\partial _z U_m(t,z) }\cdot V_m(t,z) \mod  {(t^{n+1}, U_m(t,z))}.
\end{equation*}
\end{proof}

We end up with the computation of $U_n$ using a Newton scheme. Write $U_n = U_m +  T_m$, where $T_m \in t^{m+1}\Ktzp$. By Propostion~\ref{prop:interpolatingpolynomial}, the polynomial $U_n(t,z)$ can be constructed from the following relations
\begin{equation}
\label{eq:Uh}
U_n(t , h_m (t, x_j^{(m)}(t)) + x_j^{(m)}(t) ) \mod t^{n+1}= 0, \; \text{for all }j \in \{ 1, \ldots , g\}.
\end{equation}
Equation~\eqref{eq:Uh} can be rewritten as follows:
\begin{equation}
U_n(t ,  x_j^{(m)}(t)) + h_m(t,x_j^{(m)}(t)) \cdot  \partial _z U_n (t,x_j^{(m)}) \equiv  0 \mod t^{n+1}, \; \text{for all }j \in \{ 1, \ldots , g\}.
\end{equation}
Using the fact that $ U_m(t , x_j^{(m)}(t)) = 0$ for all $j \in \{ 1, \ldots , g\}$, we get 
\begin{equation}
T_m(t ,  x_j^{(m)}(t)) + h_m(t,x_j^{(m)}(t)) \cdot  \partial _z U_m (t,x_j^{(m)}) \equiv  0 \mod t^{n+1}, \; \text{for all }j \in \{ 1, \ldots , g\}.
\end{equation}
Repeating the above calculations in the reverse direction, we obtain the next proposition.
\begin{prop}
\label{prop:Un}
Let $m \geq 0$ be an integer and $n=2m+1$. Let $T_m \in t^{m+1} \Ktzp$ be the polynomial defined by 
\begin{equation*}
T_m (t,z)=  h_m(t,z) \cdot \partial _z U_m(t,z) \mod (t^{n+1}, U_m).
\end{equation*}
Then, $U_n = U_m + T_m$ is an approximation of $U(t,z)$ modulo $t^{n+1}$. 
\end{prop}
\begin{cor}
Step~3 can be performed using at most $\softO(ng)$ operations in $K_0$.
\end{cor}
\begin{proof}
This is a direct consequence of Propositions~\ref{prop:interpolatingpolynomial} and \ref{prop:Un}. 
\end{proof}

We summarize all the steps that we performed to solve Equation~\eqref{eq:mainsystem} in Algorithm~\ref{algo:Alternantsys}.

\begin{figure}
  \begin{center}
    \parbox{0.95\linewidth}{%
      \begin{footnotesize}\SetAlFnt{\small\sf}%
        \begin{algorithm}[H]%
          \caption{Alternant Differential System Solver} %
          \label{algo:Alternantsys}%
          \SetKwInOut{Input}{Input} %
          \SetKwInOut{Output}{Output} %
          \SetKwProg{AlternantSystem}{\tt AlternantSystem}{}{}%
          \AlternantSystem{$(G,f_2,U_0,V_0,n)$}{
            \Input{$G$ $\mod t^{n}$, $f_2$ $\mod t^{n}$, $U_0(z) = \prod \limits _{j=1}^g {(z-x_j^{(0)})}$ a separable polynomial, $V_0(z)$ such that $V_0(x_j^{(0)})= y_j^{(0)}$ for all $j = 1 , \ldots g$.}
            \Output{The polynomial $U \mod t^{n+1}$ whose roots form the solution of Equation~\eqref{eq:mainsystem}} \BlankLine %
            \If{$n =0$}
            {%
            \KwRet{$ U_0 \mod t$, $1/V_0 \mod (t,U_0)$}
            }%
            {}
            $ m := \lceil \frac{n-1}{2} \rceil$\;%
			$X_m, W_m:= $ \texttt{AlternantSystem}{$(G,f_2, U_0,V_0,m)$}\;%
            {}
            $W_m:=(W_m/2)\cdot (-f_2 \cdot W_m^2 +3) \mod (t^{n+1}, U_m)$\;%
            {}
            $ V_m := f_2 \cdot W_m \mod (t^{n+1}, U_m)$\;%
            {}
            $ s_1^{(m)} , \ldots , s_{2g-1}^{(m)} :=$ \texttt{NewtonSums}$(U_m, 2g-1 , n)$\;%
            {}
            \For{$j:=1$ to $2g-1$}
            {%
            $ r_i^{(m)}:= (ds_i^{(m)}/dt)/i$\;%
            }%
            {}
            $ Hp_1^{(m)}, \ldots , Hp_g^{(m)} := \texttt{HankelProd}((r_{i+j-1}^{(m)})_{i,j} , (w_0^{(m)}, \ldots ,w_{g-1}^{(m)} ) , n)$\; 
            {}
            \For{$i:=1$ to $g$}
            {%
            $ F_i^{(m)}=  \int (G_i - Hp_i^{(m)}) \, dt \mod t^{n+1}$\;%
            }%
            {}
            $D_m := F_1^{(m)}z^g + F_2^{(m)} z^{g-1} + \cdots + F_{g-1}^{(m)} z^2 + F_g^{(m)}z \mod t^{n+1}$\;%
            Write $U_m \cdot D_m = q_{2g}^{(m)} z^{2g} + q_{2g-1}^{(m)}z^{2g-1} + \cdots + q_1^{(m)} z+ q_0^{(m)}\mod t^{n+1}$\;%
            $Q_m:= q_{2g}^{(m)} z^{g-1} + q_{2g-1 }^{(m)} z^{g-2} + \cdots + q_{g+2}^{(m)} z + q_{g+1}^{(m)} \mod t^{n+1}$\;%
	    $T_m := -Q_m \cdot V_m \mod (t^{n+1}, U_m)$\;
          \KwRet{$U_m + T_m \mod t^{n+1}, W_m \mod t^{n+1}$}
          }
        \end{algorithm}
      \end{footnotesize}
    }
  \end{center}
\end{figure}

\begin{thm}
\label{thm:mainthmbis2}
Let $K_0$ be an unramified extension of $\mathbb{Q}_p$. Let $n, g \in \mathbb{N}$, $N \in \mathbb{N}^*$, $G \in \OKtp ^g$, $f_2 \in \OKtzp$ of degree $O(g)$, $U_0 =  \prod \limits _{j=1}^g {(z-x_j^{(0)})} \in \OKp[z]$, $V_0 \in \OKp [z]$ of degree $g-1$ such that $V_0(x_j^{(0)}) \mod p \neq 0$ and $U_0$ divides $f_2 - V_0^2$ in $\OKtzp$. We assume that $U_0$ is separable over the the residue field of $K_0$ and that the polynomial $U(t,z)$, whose roots form a solution of Equation~\eqref{eq:mainsystem}, have coefficients in $\OKtp$. Then, the procedure \texttt{AlternantSystem} runs with fixed point arithmetic at precision $O(p ^{M+ \lfloor \log _p (2g-1) \rfloor})$, with $M= \max (N,2) + \lfloor \log _ p(n) \rfloor$ if $p=3$ and $M = N + \lfloor \log _ p(n) \rfloor$ otherwise. All the computations are done in $\OKp$ and the result is correct at precision $O(p^N)$.

\end{thm}

\begin{proof}
The proof is similar to the proof of Theorem~\ref{thm:mainthmbis}. 
\end{proof}

\begin{prop}
\label{prop:complexityalternant}
When performed with fixed point arithmetic at precision $O(p^M)$, the bit complexity of Algorithm~\ref{algo:Alternantsys} is $\softO \left( ng \cdot \cA(K_0;M) \right)$, where $\cA (K_0;M)$ denotes an upper bound on the bit complexity of the arithmetic operations in $\OKp / p^M \OKp$.
\end{prop}

\begin{proof}
Let $\cD$ denote the algebraic complexity of Algorithm~\ref{algo:Alternantsys}, then we have the following relation
\begin{equation*}
\cD (n)  \leq  \cD \left( \left\lceil \dfrac{n-1}{2} \right\rceil \right) + \softO \left( n g \right).
\end{equation*}
Solving the recurrence, we find $ \cD (n) = \softO \left( ng \right)$. Therefore, the bit complexity of Algoritmh~\ref{algo:Alternantsys} is $\softO \left ( ng \cdot \cA(K_0;M) \right )$. 
\end{proof}

 \section{Fast computation of the multiplication by-$\ell$ maps}
\label{subsec:Multbyell}

Thanks to Algorithm~\ref{algo:Alternantsys}, we now have fast algorithms for computing rational representations of separable isogenies between Jacobians of hyperelliptic curves defined over fields of odd characteristic, after having possibly lifted the two curves and the normalization matrix to the $p$-adics. In this section, we are interested in the computation of a rational representation of the multiplication by an integer. It is therefore necessary to know in advance some bounds on the degrees of its components. 

It has been proved that the degrees of the components of a rational representation of the multiplication-by-$\ell$ are bounded by $O(\ell ^2)$, only for curves of genus $2$ and $3$~\cite[Chapter~4]{abela18}. In the general case, it has been shown that these degrees are bounded by $O_g(\ell ^3)$\footnote{the notation $O_g$ means that we are hiding the terms that depend on $g$}~\cite[Theorem~4.13]{abela18}, although experiments show that they are only quadratic in $\ell$. 

In Section~\ref{subsec:cantordivisionpolynomials}, we use the results of Section~\ref{subsec:RationalRepresentation} to reduce the bound to $O(g\ell^2)$.  Consequently, we derive from Algorithm~\ref{algo:Alternantsys} a quasi-optimal algorithm to compute a rational representation of the multiplication by an integer.

\subsection{Cantor $\ell$-division polynomials}
\label{subsec:cantordivisionpolynomials}
Let $C : y^2 = f(x)$ be a hyperelliptic curve of genus $g$ over a finite field $k$ and $\ell >g$ an integer coprime to the characteristic of $k$.\\
Let $P \in C(k)$. For a generic point $Q=(x,y)$ on $C$, the Mumford representation of the element $\ell [Q - P]$ in the Jacobian of $C$ can be written as follows
$$
\ell [Q- P] = \left ( X^g + \sum \limits _{i=1}^{g-1} {\dfrac{d_i(x)}{d_g(x)}} X^i, \,   y \sum \limits _{i=1}^{g-1} {\dfrac{e_i(x)}{e_g(x)}} X^i \right ),
$$
where the numerators $d_0, \ldots , d_{g-1}, e_0, \ldots , e_{g-1}$ are polynomials in $k[x]$ and the denominators $d_g$ and $e_g$ are monic polynomials in $k[x]$. Therefore, $\left ( \frac{d_0}{d_g}, \ldots ,  \frac{d_{g-1}}{d_g} ,  \frac{e_0}{e_g}, \ldots ,  \frac{e_{g-1}}{e_g} \right )$ is a rational representation of the multiplication-by-$\ell$ map.
\begin{defi}
The $2g + 2$ polynomials $d_0, \ldots , d_g, e_0, \ldots e_g$ are called Cantor’s $\ell$-division polynomials.
\end{defi}
Since the multiplication-by-$\ell$ endomorphism is a separable $(\ell^2 , \ldots , \ell ^2)$-isogeny, we can then apply Propositions~\ref{prop:degreerationalfractions} and \ref{prop:degreegl} and Remark~\ref{remark:degreegeneralbound} in order to obtain bounds on the degrees of the Cantor’s $\ell$-division polynomials. This gives the following result.

\begin{prop}
\label{prop:degreeboundmult}
The degrees of the polynomials $d_0 , \ldots, d_g$ are bounded by $g\ell^2$. Moreover,
\begin{itemize}
\item if $P$ is a Weierstrass point, then the degrees of $e_0 , \ldots, e_g$ are bounded by
$$
\left \{  \begin{array}{cc}
\frac{3}{2} g \ell^2 + g + 1 & \text{if } \deg (f) = 2g+1\\
2 g \ell^2 + g + 1 & \text{otherwise}
\end{array} \right.
$$
\item if $P$ is not a Weierstrass point, then the degrees of $e_0 , \ldots, e_g$ are bounded by
$$
\left \{  \begin{array}{cc}
3g \ell^2 + g + 1 & \text{if } \deg (f) = 2g+1\\
4 g \ell^2 + g + 1 & \text{otherwise}
\end{array} \right.
$$

\end{itemize}
\end{prop}

\begin{remark}
The bounds obtained in Propostion~\ref{prop:degreeboundmult} are not optimal. In fact, the experiments carried out by Abelard in his thesis~\cite[{Section~4.2}]{abela18} show that Cantor $\ell$-division polynomials have degrees slightly smaller than the bounds that we have obtained in Propostion~\ref{prop:degreeboundmult}.
\end{remark}
The next theorem and its proof give an efficient algorithm to compute Cantor's division polynomials. 

 \begin{thm}
 \label{thm:multiplicationbyellmap}
 Let $p$ an odd prime number and $g>1$ an integer. Let $\ell$ be an integer greater than $g$ and coprime to $p$. Let $C:y^2 = f(x)$ be a hyperelliptic curve of genus $g$ defined over a finite field $k$ of odd characteristic $p$. There exists an algorithm that computes Cantor $\ell$-division polynomials of $C$, performing at most $\softO(\ell^2  g^2 )$ operations in $k$.
 \end{thm}

 \begin{proof}
 Let $p$ be the characteristic of $k$ and $d=[k : \mathbb F_p]$. For the sake of simplicity, we will assume that $C$ admits a Weierstrass point over $k$. The algorithm performs the following steps:
 \begin{enumerate}
 \item Pick a Weierstrass point $P \in C(k)$.  
 \item Chose a point $Q\in C(k)$ different from $P$ such that $\ell [ Q - P]$ is generic.
 \item Lift $C$ arbitrarily as $\widetilde{C} : y^2= \tilde{f}(x) $ over an unramified extension $K_0$ of $\mathbb{Q}_p$ of degree $d$ with a $p$-adic precision equal to $1+ \lfloor \log _p (2g \ell^2) \rfloor + \lfloor \log _p (2g-1) \rfloor $.
 \item Lift $P$ as $\tilde{P}$ and $Q$ as $\tilde{Q}$ over $K_0$ such that $\tilde{P}, \tilde{Q} \in \tilde{C}(K_0)$.

 \item Solve the differential equation~\eqref{eq:mainsystem} by applying  Algorithm~\ref{algo:Alternantsys} to the following input:
 \begin{itemize}
 \item $n = 2g \ell^2$,
 \item $f_2  = \tilde{f}$,
 \item $U_0(z)$: the first Mumford coordinate of $\ell [ \tilde{Q} - \tilde{P}]$,
\item $V_0(z)$: the second Mumford coordinate of $\ell [ \tilde{Q} - \tilde{P}]$,
\item $G(t)$, given by the following relation 
\begin{equation*}
G = \dfrac{\ell}{v(t)} \begin{pmatrix} 1 \\ u(t) \\  u(t)^2 \\ \vdots \\ u(t)^{g-1} \end{pmatrix},
\end{equation*}
 where $u(t) = t + x_{\tilde{Q}}$ and $v(t) = \sqrt{f(u(t))}$ such that $v(0) = y_{\tilde{Q}}$.
 \end{itemize}
 Let $U(t,z)$ be the reduction of the output of the algorithm in $k$.
\item Reconstruct from $U(t,z)$ the $g+1$ polynomials $d_0, \ldots, d_{g}$.
\item Recover the polynomials $e_0, \ldots , e_g$ from $d_0, \ldots, d_{g}$ and the equation of $C$.
 \end{enumerate}
 The time complexity of the algorithm depends mainly on the complexity of steps~5,6 and 7. According to Proposition~\ref{prop:complexityalternant}, step~5 can be carried out for a cost of $\softO(g^2  \ell^2 )$ operations in $k$. The $g+1$ polynomials $d_0, \ldots, d_{g}$ are obtained by reconstructing (for example) $d_0/d_g$ using Pad\'e approximants from the constant coefficient of $U$ then multiplying the other coefficients of $U$ by $d_g$ to recover $d_1, \ldots , d_{g-1}$. Therefore, step~6 requires $\softO(g^2 \ell ^2 )$ operations in $k$ as well. Step~7 is executed as follows: we make use of the polynomials  $d_0, \ldots, d_{g}$ to increase the $t$-adic approximation of the polynomial $U(t,z)$ to $2\deg (e_0)$. We compute, using a Newton iteration, the degree $g$ polynomial $V(t,z)$ such that $V(0,z) = V_0(z)$ and $$V(z,t)^2 \equiv f(z) \mod (t^{2\deg(e_0)}, U(t,z)).$$  We reconstruct (for example) the rational fraction $e_0/e_g$ from the constant coefficient of $V$. The polynomials $e_1 , \ldots , e_{g-1}$ are obtained by multiplying $e_g$ with the non-constant coefficients of $V$. This can also be carried out using $\softO(g^2\ell^2)$ operations in $k$.
 \end{proof}

\subsection{Experiments}
\label{subsec:experimentsmult}
We made an implementation of both Algorithm~\ref{algo:Alternantsys} and the Padé approximant
step using the \textsc{half-gcd} algorithm given in~\cite{thome03} with the \textsc{magma} computer algebra
system~\cite{magma} to compute Cantor $\ell$-division polynomials in $\mathbb{F}_5$ for hyperelliptic curves. Our implementation is available at~\cite{github} . 
Timings are detailed in Figures~\ref{fig:timingsCantor} and~\ref{fig:timingsg}. All the calculations were done in the ring $\mathbb{Z}_5$ with a fixed precision which is equal to $1 + \lfloor \log _5 (2g \ell^2)  \rfloor + \lfloor \log _5(2g-1) \rfloor.$
The observed timings fit rather well with the expected time complexity, which is $\softO(\ell ^2 g^2  )$: Figure~\ref{fig:timingsg} (resp. Figure~\ref{fig:timingsCantor}) shows that the time complexity of our algorithm is almost linear in $g^2$ (resp. in $\ell^2$).

\begin{figure*}[http]

    \centering
    \begin{subfigure}[t]{0.5\textwidth}
        \centering
        \includegraphics[width=.62\textwidth, angle=-90]{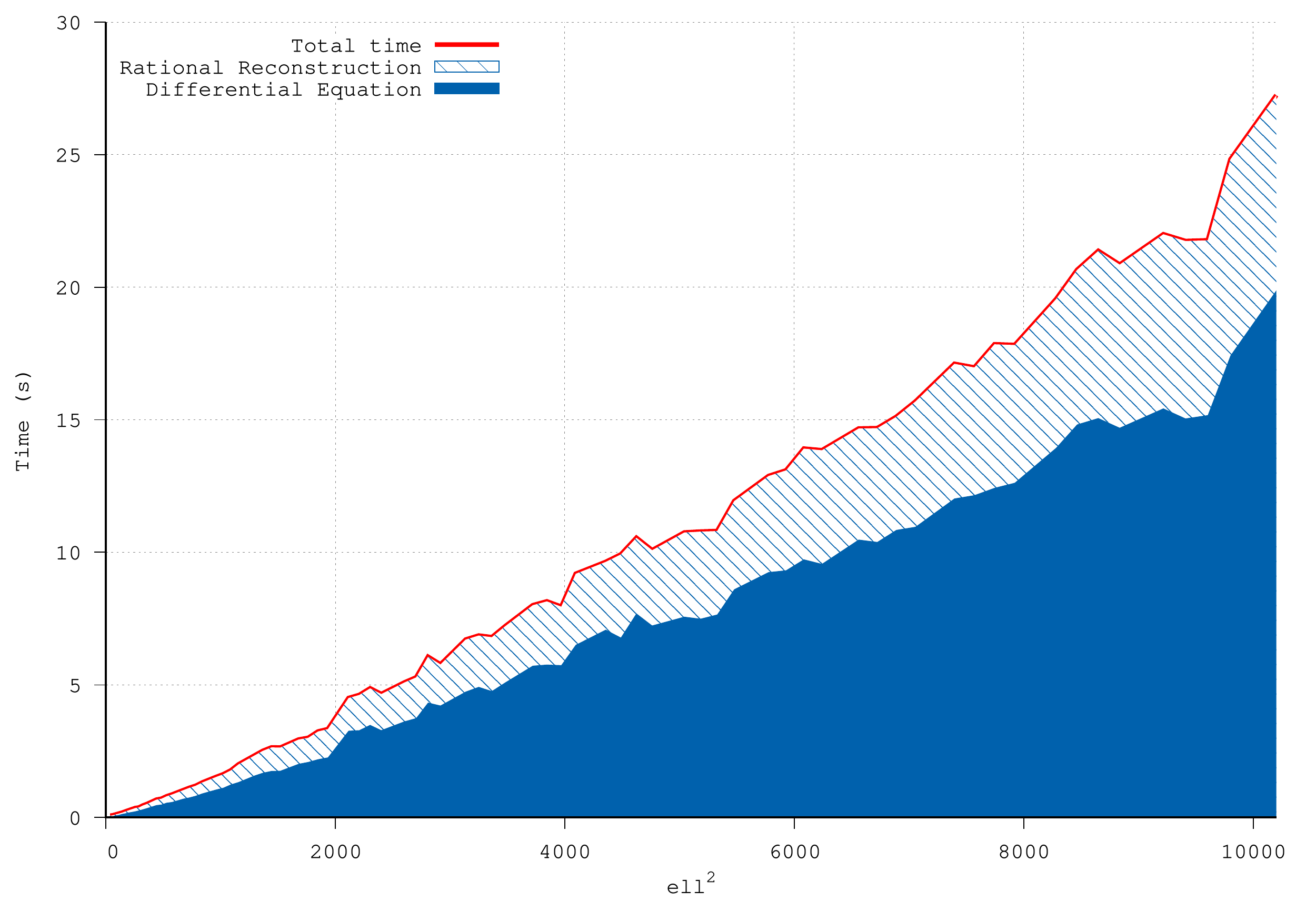}
        \caption{$g=4$}
    \end{subfigure}%
    ~
    \begin{subfigure}[t]{0.5\textwidth}
        \centering
        \includegraphics[width=.62\textwidth, angle=-90]{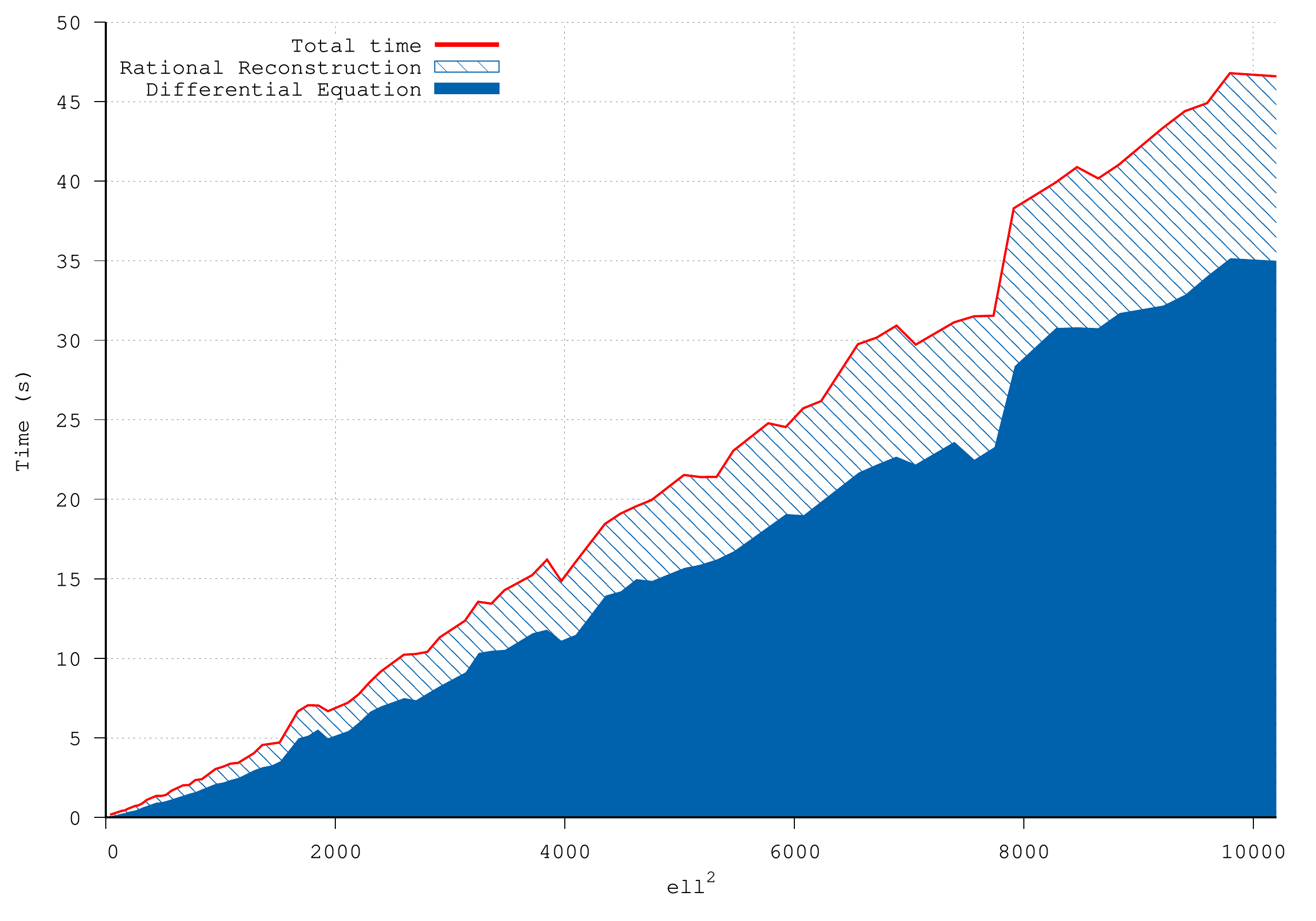}
        \caption{$g=5$}
    \end{subfigure}
    ~
    \newpage
    \centering
     \begin{subfigure}[t]{0.5\textwidth}
        \centering
        \includegraphics[width=.62\textwidth, angle=-90]{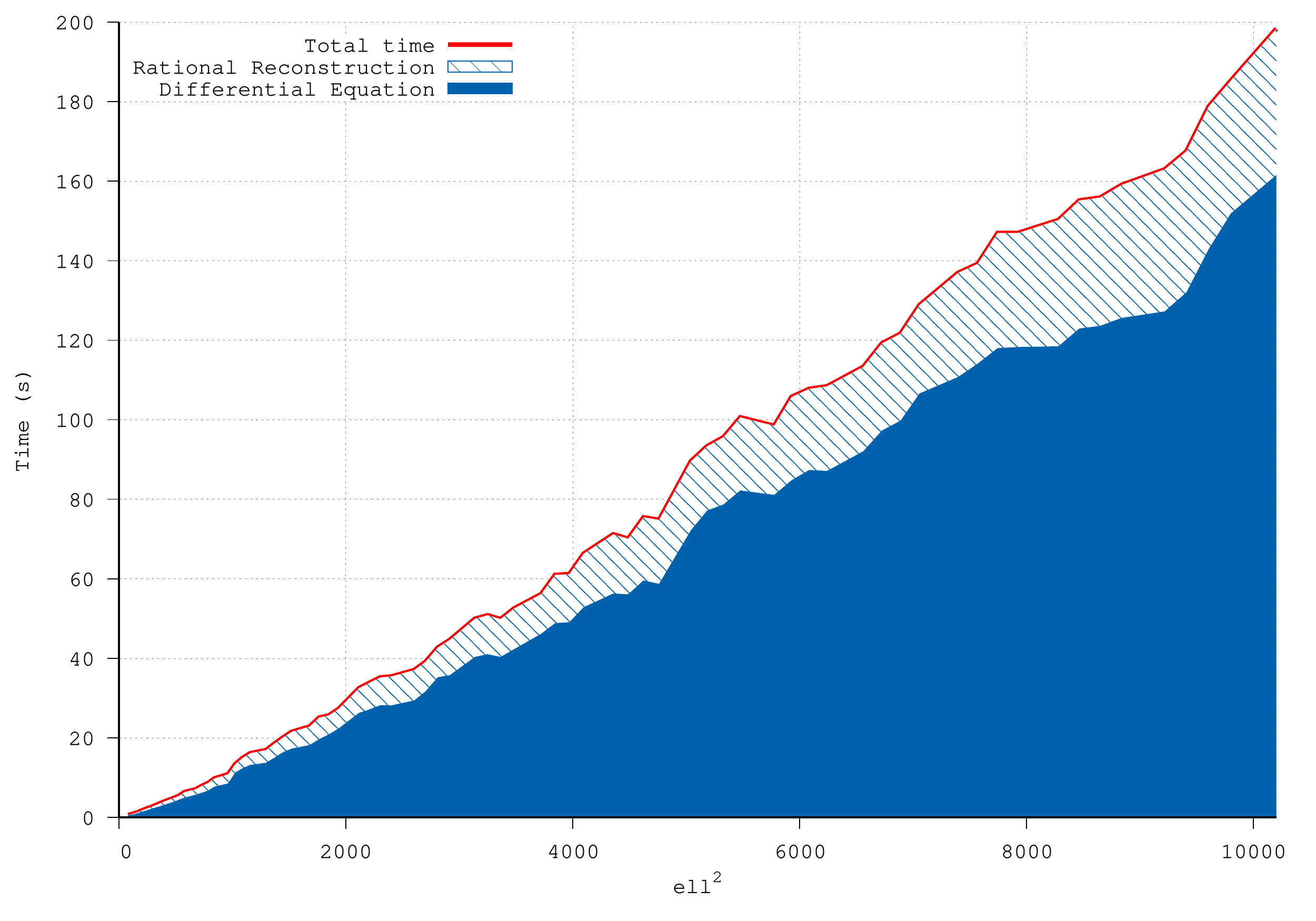}
        \caption{$g=8$}
    \end{subfigure}%
    ~
    \begin{subfigure}[t]{0.5\textwidth}
        \centering
        \includegraphics[width=.62\textwidth, angle=-90]{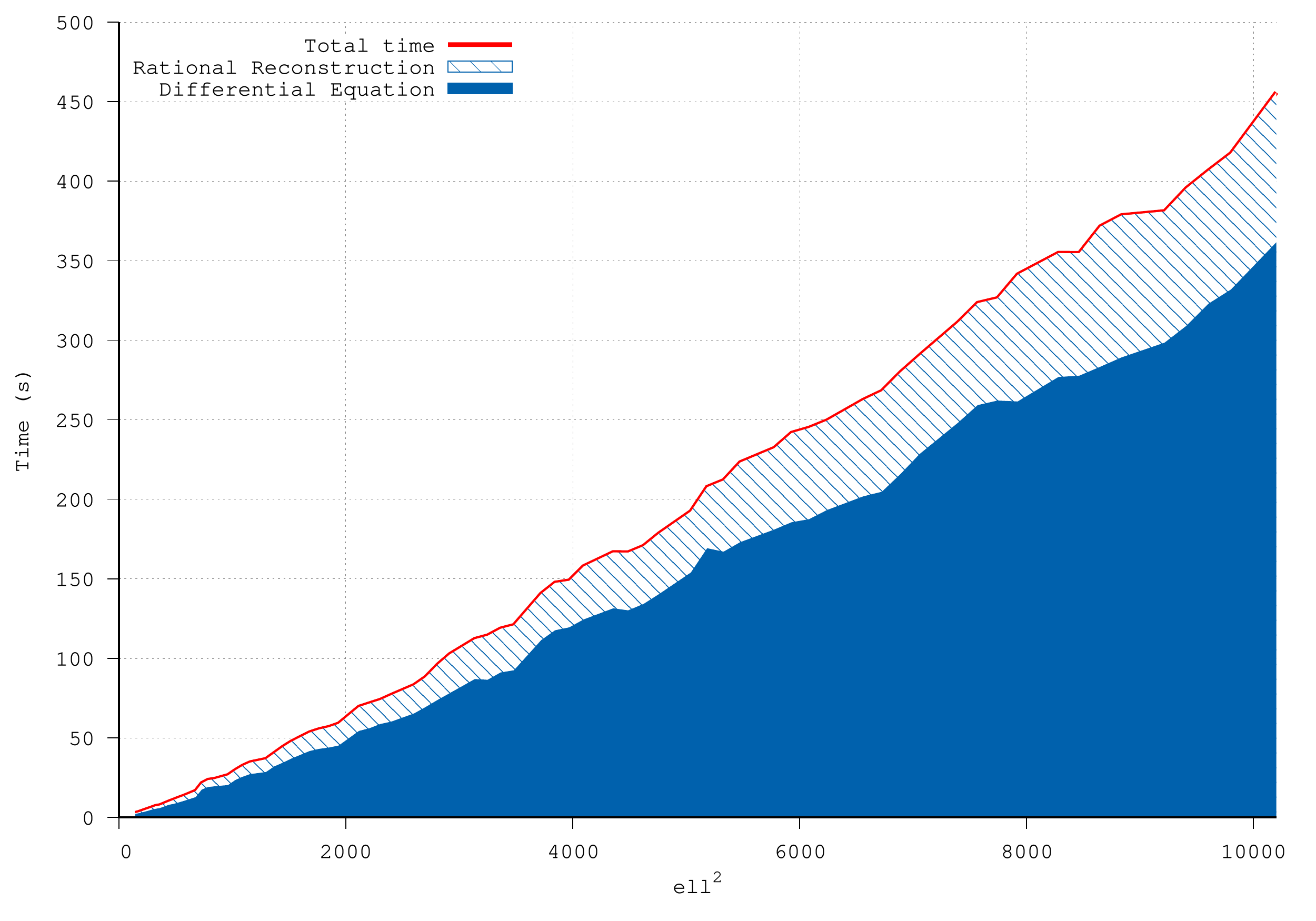}
        \caption{$g=11$}
    \end{subfigure}

    \medskip

  {\scriptsize%
  Timings obtained with \textsc{magma
  V2.25-7} on a laptop with an \textsc{intel} processor
  \textsc{E5-2687WV4@3.00ghz}}

    \caption{Computation of the multiplication-by-$\ell$ map over $\mathbb{F}_5$ for $\ell \in \{g+1 , \ldots , 101\}$ and such that gcd$(\ell , 5)=1$ \label{fig:timingsCantor} }
\end{figure*}

\begin{figure*}[http]

    \centering
    
        \includegraphics[width=.4\textwidth, angle=-90]{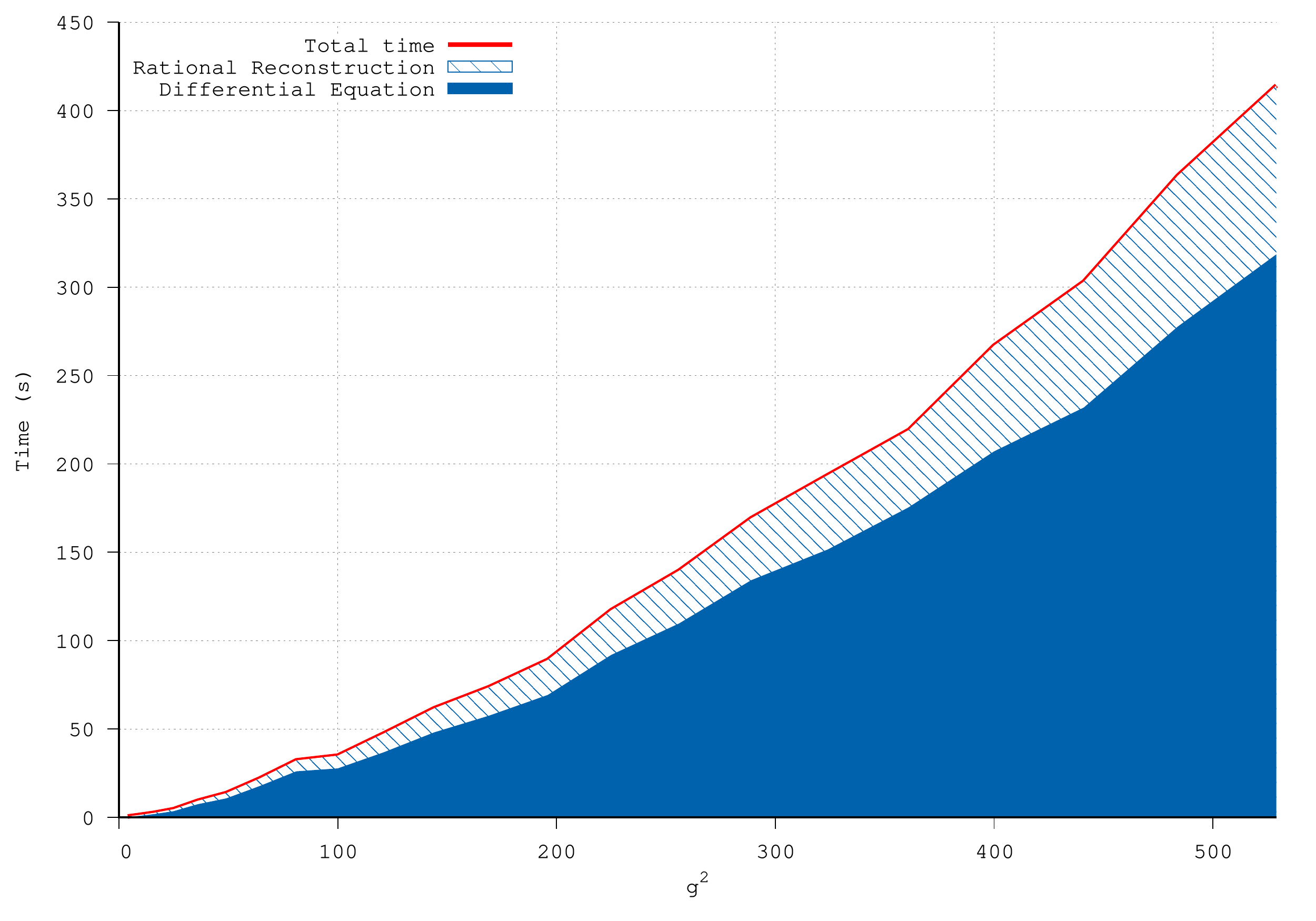}
           
    \medskip

  {\scriptsize%
  Timings obtained with \textsc{magma
  V2.25-7} on a laptop with an \textsc{intel} processor
  \textsc{E5-2687WV4@3.00ghz}}

    \caption{Computation of the multiplication-by-$31$ map over $\mathbb{F}_5$ for $g \in \{2 , \ldots , 30\}$.
     \label{fig:timingsg} }
\end{figure*}

\bibliographystyle{elsarticle-num}

\bibliography{synthbib.bib}

\end{document}